\documentclass[11pt,a4paper,reqno]{amsart}

\usepackage{graphicx, amsmath, amsthm, amssymb, enumerate, esint, mathtools}
\usepackage[shortlabels]{enumitem}
\usepackage{amsfonts,mathrsfs,color,multirow}
\usepackage{pgf,tikz}
\usetikzlibrary{arrows}
\usepackage[toc,page]{appendix}
\graphicspath{ {./ImagesForFigures/} }
\numberwithin{equation}{section}

\setlist[itemize]{noitemsep, topsep=0pt}

\newcommand{\R}[1]{\mathbf{R}^{#1}}
\newcommand{\N}{\mathbf{N}}
\newcommand{\Mink}{\mathbf{M}^{3}}
\newcommand{\Hyp}{\mathcal{H}}
\newcommand{\sgn}[1]{\text{sgn}(#1)}
\newcommand{\ep}{\varepsilon}

\newcommand{\dsum}[3]{\displaystyle \sum_{#1}^{#2}#3}

\newcommand{\fourbyfour}[9]{\left( \begin{array}{cccc}
#1 & #2 & \cdots & #3 \\
#4 & #5 & \cdots & #6\\
\vdots & \vdots & \cdots  & \vdots \\
#7 & #8 & \cdots & #9 \end{array} \right)}
\newcommand{\ip}[2]{\left<  #1, #2 \right>}
\newcommand{\diag}{\text{diag}}

\numberwithin{equation}{section}

\theoremstyle{plain}
\newtheorem{theorem}{Theorem}[section]
\newtheorem{prop}[theorem]{Proposition}
\newtheorem{lemma}[theorem]{Lemma}  
\newtheorem{cor}[theorem]{Corollary}

\theoremstyle{definition}

\newtheorem{defn}[theorem]{Definition}
\newtheorem{example}[theorem]{Example}

\theoremstyle{remark}
\newtheorem{remark}[theorem]{Remark}

\allowdisplaybreaks

\title[Confocal Curves]{Pseudo-Euclidean Billiards within Confocal Curves on the Hyperboloid of One Sheet}
\author[S. Gasiorek]{Sean Gasiorek*}
\thanks{* Corresponding author.}
\email{sean.gasiorek@sydney.edu.au}
\address{School of Mathematics and Statistics, Carslaw Building F07, University of Sydney, NSW 2006, Australia} 

\author[M. Radnovi\'c]{Milena Radnovi\'c}
\email{milena.radnovic@sydney.edu.au}
\address{School of Mathematics and Statistics, Carslaw Building F07, University of Sydney, NSW 2006, Australia
\newline
Mathematical Institute SANU, Belgrade, Serbia
} 
%% Overleaf shows 1991 in the document here, but when compiled outside of Overleaf it properly shows 2020 as the year. 

\begin{document}

\begin{abstract}
We consider a billiard problem for compact domains bound\-ed by confocal conics on a hyperboloid of one sheet in the Minkowski space. 
We show that there are two types of confocal families in such setting.
Using an algebro-geometric integration technique, we prove that the billiard within generalized ellipses of each type is integrable in the sense of Liouville. 
Further, we prove a generalization of the Poncelet theorem and derive Cayley-type conditions for periodic trajectories and explore geometric consequences. 
\end{abstract}

\maketitle

\tableofcontents

%MSC2010 14H70, 41A10, 70H06, 37J35, 26C05
%MSC2020 14H70, 70H06, 70H12, 37J35, 37J45

%Keywords: Billiards, Minkowski space, elliptic billiards, confocal quadrics, periodic trajectories, hyperboloid

\section{Introduction}\label{intro}

Mathematical billiard is a dynamical system where a particle moves freely within a domain and obeys the billiard law, where the angle of incidence equals angle of reflection, off the boundary \cite{Bir,KozlovTr}.
The behavior of such a mechanical system is dependent upon the geometric properties of the boundary and of the underlying space. 

It is well known that classical theorems of Jacobi, Chasles, Poncelet, and Cayley 
and their generalizations imply integrability and many beautiful geometric properties of the billiards within confocal quadrics in Euclidean and pseudo-Euclidean spaces of arbitrary finite dimension, see e.g. \cite{GKT,KT,DR3,DR4}.

%The Jacobi theorem states that through every point there are $d$ confocal quadrics which are pairwise orthogonal at this point; the Chasles theorem states that every line is tangent to $d-1$ confocal quadrics whose tangent hyperplanes at the points of tangency are pairwise orthogonal \cite{Ber}. The Poncelet theorem and its analogues in higher dimensions describe a property of closed trajectories in ellipsoids: there exists a closed trajectory with $d-1$ confocal caustics if and only if infinitely many such trajectories  exist and all of them have the same period. Pseudo-Euclidean versions of these theorems have also been established along with explorations of confocal billiards in Minkowski spaces \cite{DR3,DR4,GKT,KT}. Such pseudo-Euclidean spaces have also been a geometric model for special and general relativity \cite{CGP}.

%The Poncelet theorem first drew the attention of Cayley, who proved an analytic condition for caustic conics of the billiard in the planar Euclidean case \cite{C1,C2}. A proof using techniques from algebraic geometry was shown in \cite{GH} and the generalization to arbitrary dimension $d$ in \cite{DR1,DR2}. This generalization utilized the $L-A$ matrix factorization technique of Moser and Veselov \cite{MV} who applied this to several discrete dynamical systems including the billiard in the ellipsoid. Veselov also used these ideas to the billiard on the Lobachevsky space and the sphere \cite{V}. An alternate approach to a Poncelet-type result on an ellipsoid is given by \cite{GKT} by analyzing the geodesic flow on the ellipsoid in Minkowski space.

The main inspiration for our work comes from the paper \cite{V}, where billiards within confocal families on a hyperboloid of two sheets were considered. There, the restriction of the Minkowski metric gives rise to the geometry of Lobachevsky.
In this paper, we will study confocal families and the corresponding billiards on a hyperboloid of one sheet.
The restriction of the metric will be Lorentzian in our case, and we get novel and intriguing geometric and dynamical properties of generalized elliptical billiards. The related topic of geodesic scattering on the hyperboloid of one sheet with the Euclidean metric has only recently been studied \cite{VW}.

In Section \ref{HyperboloidIntro} we introduce the three-dimensional Minkowski space and its properties. We also define and explore confocal families of conics as intersections of confocal families of cones with the hyperboloid of one sheet and analyze their properties.
Interestingly enough, this leads to two different types of confocal families, and we give geometric description for each of them. 
%Two distinct billiard systems within compact domains are possible which we call the collared $\Hyp$-ellipse and transverse $\Hyp$-ellipse. 
Section \ref{BonH} introduces billiards in Minkowski space and discusses basic properties of geodesics on the hyperboloid of one sheet. In Section \ref{FMMP} we provide a review of the relevant matrix factorization technique, adapted from \cite{V}, which is applied to these two scenarios and proves the explicit integrability of the two billiard systems. 
It is interesting to note that the Lax pair will produce a discrete trajectory for any initial conditions, even in the cases when the consecutive points cannot be connected by geodesics.
Section \ref{SpectralCurves} addresses the Cayley condition for periodic orbits and addresses a Poncelet-like theorem. Lastly, section \ref{GC} addresses the unique geometric properties of each billiard system and provides examples.

\section{Confocal Families on the Hyperboloid of One Sheet}
\label{HyperboloidIntro}

The \emph{three-dimensional Minkowski space} $\Mink$ is the real 3-dimensional vector space $\R{3}$ with the symmetric nondegenerate bilinear form
\begin{equation}
\ip{x}{y} = - x_0 y_0 + x_1 y_1  +  x_2 y_2.
\end{equation}
%We note several important properties of the geometry of the Minkowski space $\Mink$. The \emph{Minkowski distance} between two points $x, y \in \Mink$ is $dist(x,y) = \sqrt{\ip{x-y}{x-y}}$. The inner product can produce negative values, and in such a case we choose the distance to be the positive square root of the imaginary part. 

\begin{defn}\label{SLT}
For a vector $v$, we say that it is:
	\begin{itemize}
		\item \emph{space-like} if $\ip{v}{v}>0$ or $v=0$;
		\item \emph{light-like} if $\ip{v}{v}=0$ and $v \neq 0$;
		\item \emph{time-like} if $\ip{v}{v}<0$.
	\end{itemize}
	Two vectors $u$ and $v$ in Minkowski space are \emph{orthogonal} if $\ip{u}{v}=0$. Note that any light-like vector is orthogonal to itself. 
A line $\ell$ will be called \emph{space-like}, \emph{light-like}, or \emph{time-like} if such is its direction vector.
\end{defn}

We take interest in the hyperboloid of one sheet
\begin{equation}
\Hyp\ :\ \ip{x}{x} =1
\end{equation}
in $\Mink$. The metric $$ds^2 = -dx_0^2 + dx_1^2 + dx_2^2$$ restricted to $\Hyp$ is a Lorentz metric of constant curvature. Geodesics of this metric are the intersections of $\Hyp$ and planes through the origin, also called \emph{central planes}.
Such intersections can take the form of plane ellipses, hyperbolas, or straight lines. We call these geodesics \emph{space-}, \emph{time-}, and \emph{light-}like, respectively, as the tangent vectors to these geodesics obey the inequalities stated in Definition \ref{SLT}. 

The hyperboloid of one sheet is not geodesically connected even though it is a geodesically complete Lorentzian manifold \cite{ON,Bee}. However, we can state specifically when and how two points of $\Hyp$ can be connected by a geodesic, provided the two points are not \emph{antipodal}, i.e. distinct points on $\Hyp$ that are on the same line through the origin.

\begin{prop}[\cite{ON}]%Prop 5.38 of ONeill, pg 149-50
Let $p$ and $q$ be distinct nonantipodal points on the hyperboloid of one sheet, $\Hyp$, and let $\ip{\cdot}{\cdot}$ be the Minkowski inner product. 
\begin{enumerate}[(i)]
    \item If $\ip{p}{q}>1$, then $p$ and $q$ lie on a unique time-like geodesic that is one-to-one;
    \item If $\ip{p}{q}=1$, then $p$ and $q$ lie on a unique light-like geodesic;
    \item If $-1 < \ip{p}{q}<1$, then $p$ and $q$ lie on a unique space-like geodesic that is periodic;
    \item If $\ip{p}{q} \leq -1$, then $p$ and $q$ cannot be connected by a geodesic. 
\end{enumerate}
\label{GeodesicProp}
\end{prop}

\begin{remark}
Suppose $p$, $q$ are antipodal points on $\Hyp$, and write $q=-p$. Then there is a family of planes containing the line through $p$ and $-p$, and as such there is no unique geodesic connecting antipodal points. In fact, infinitely many space-like geodesics connect a point and its antipode. 
\label{Antipodal}
\end{remark}

Consider a cone in $\Mink$
\begin{equation}\label{eq:coneA}
\ip{Ax}{x} = 0
\end{equation}
and its dual cone 
\begin{equation}\label{eq:coneAInv}
\ip{A^{-1}x}{x} = 0
\end{equation}
for a matrix $A$ satisfying $\ip{Ax}{y} = \ip{x}{Ay}$.
In general, the matrix $A$ is not diagonalizable over $\R{}$. 
However, when the curves of intersection of the hyperboloid $\Hyp$ and the cone (\ref{eq:coneAInv}) bound a compact domain on $\Hyp$ then $A^{-1}$ (and therefore $A$) is diagonalizable, which we prove in the following proposition.

\begin{prop}\label{prop:Adiag}
Suppose that all points of the cone (\ref{eq:coneAInv}), apart from its vertex, satisfy the inequality $\ip{x}{x}>0$.
%, which corresponds to either of the two geometric cases shown in figure \ref{TwoCases}. 
Then $A^{-1}$ is diagonalizable in some orthogonal coordinate system. 
\label{Diagonalizable}
\end{prop}

\begin{proof}
The proof is similar to the proof of Proposition 1 in \cite{V}. 
Consider the function $f(x) = \ip{A^{-1}x}{x}/\ip{x}{x}$. 
The function $f$ is well-defined on $\Hyp$ and vanishes on $\mathcal{C}$, the curves of intersection of the cone $\ip{A^{-1}x}{x}=0$ and $\Hyp$. 
The cone $\ip{x}{x}=0$ is the asymptotic cone to $\Hyp$, so any cone $\ip{A^{-1}x}{x}=0$ whose points satisfy $\ip{x}{x}>0$ must bound one or two compact domains on $\Hyp$.
As such, $f$ must have a maximum or minimum at some point $x_0$ of the domain bounded by $\mathcal{C}$. At this point we have $f^\prime(x_0) =0$ or $$A^{-1}x_0 = \lambda_0 x_0, \qquad \lambda_0 = \ip{A^{-1}x_0}{x_0}/\ip{x_0}{x_0},$$ so that $x_0$ is an eigenvector of $A^{-1}$. In the orthogonal complement of $x_0$, $W = \{ x \in \Mink \; : \; \ip{x}{x_0}=0\}$, we have two quadratic forms, namely the restrictions of $\ip{A^{-1}x}{x}$ and $\ip{x}{x}$. The second is positive definite and the result follows from the spectral theorems.
\end{proof}

In light of Proposition \ref{prop:Adiag}, the curves of intersection of the cone $\ip{A^{-1}x}{x}=0$ and $\Hyp$ bound either one compact domain and two unbounded domains, or two compact domains and one unbounded domain. We can describe when each of these cases occur in terms of the the entries of $A = \diag(a_0,a_1,a_2)$. 

\begin{defn}\label{ElHyp}
If the cone $\ip{A^{-1}x}{x}=0$ divides $\Hyp$ into one compact domain and two unbounded domains, we call the boundary curves a \emph{collared $\Hyp$-ellipse}. In some orthogonal coordinate system in $\Mink$ the collared $\Hyp$-ellipse is determined by the equation 
\begin{equation}
-\frac{x_0^2}{a_0}+ \frac{x_1^2}{a_1}  + \frac{x_2^2}{a_2}=0
\label{Eqn2}
\end{equation}
with $0 < a_0 < a_1 < a_2$. If the cone $\ip{A^{-1}x}{x}=0$ divides $\Hyp$ into two compact domains and one unbounded domain, we call the boundary curves a \emph{transverse $\Hyp$-ellipse}. In some orthogonal coordinate system in $\Mink$ the transverse $\Hyp$-ellipse is determined by equation \ref{Eqn2} with $a_1<0 < a_0 < a_2$.
\end{defn}

%{\color{blue}
%There is no good justification for giving the name "hyperbola" to a bounded conic.
%The argument that such a curve is projected into a hyperbola using certain transformation is not good enough, since that is true for any ellipse. For example, an ellipse in the plane will be transformed to a hyperbola if we embed the plane to a projective plane and then take the line at the infinity to intersect that ellipse.
%The application of the Klein coordinates to the present situation is, in fact, a complete analogue to that.
%}

\begin{figure}[tbhp]
\begin{tabular}{c c}
a) \includegraphics[width=0.43\textwidth]{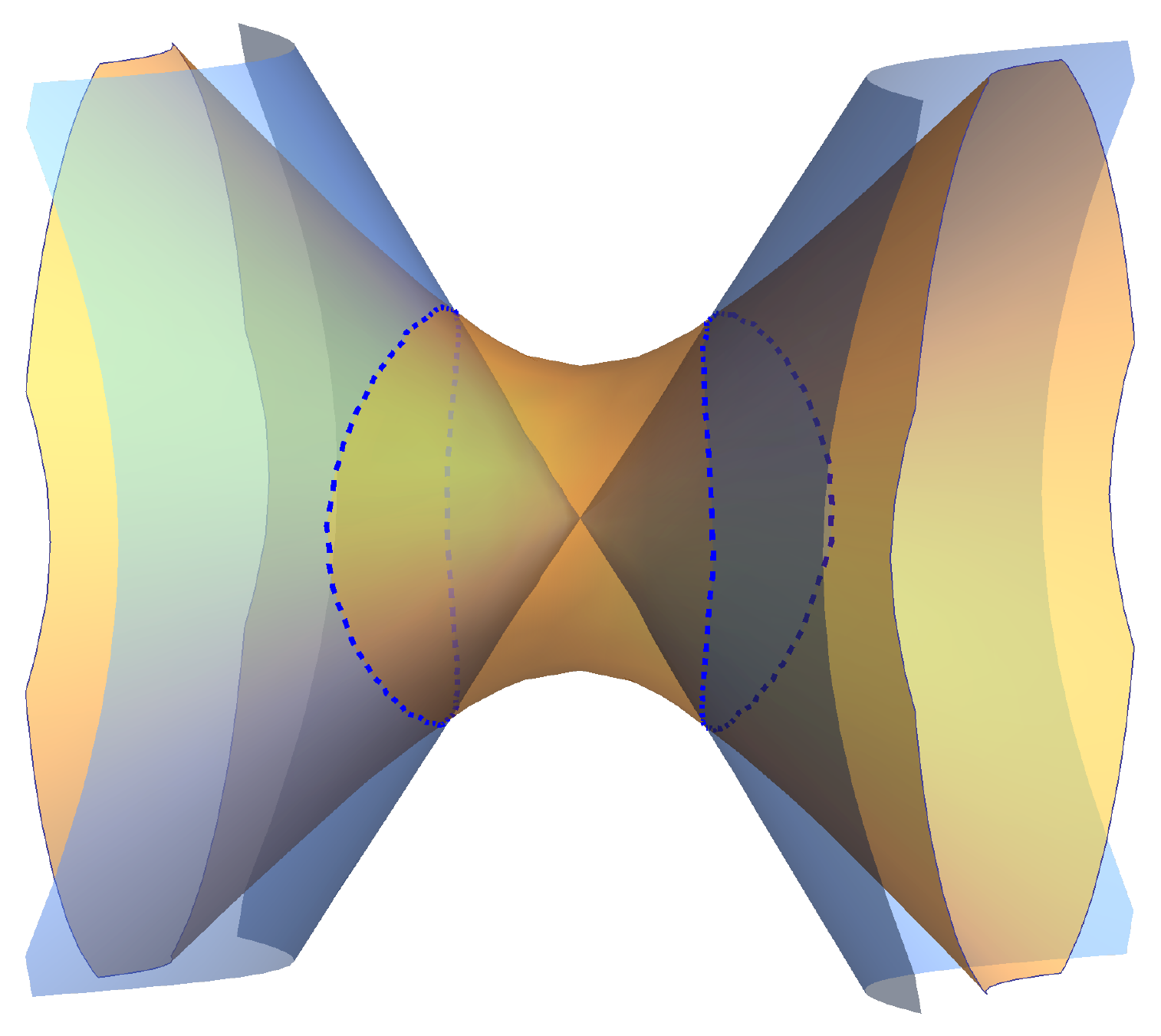} & b) \includegraphics[width=0.43\textwidth]{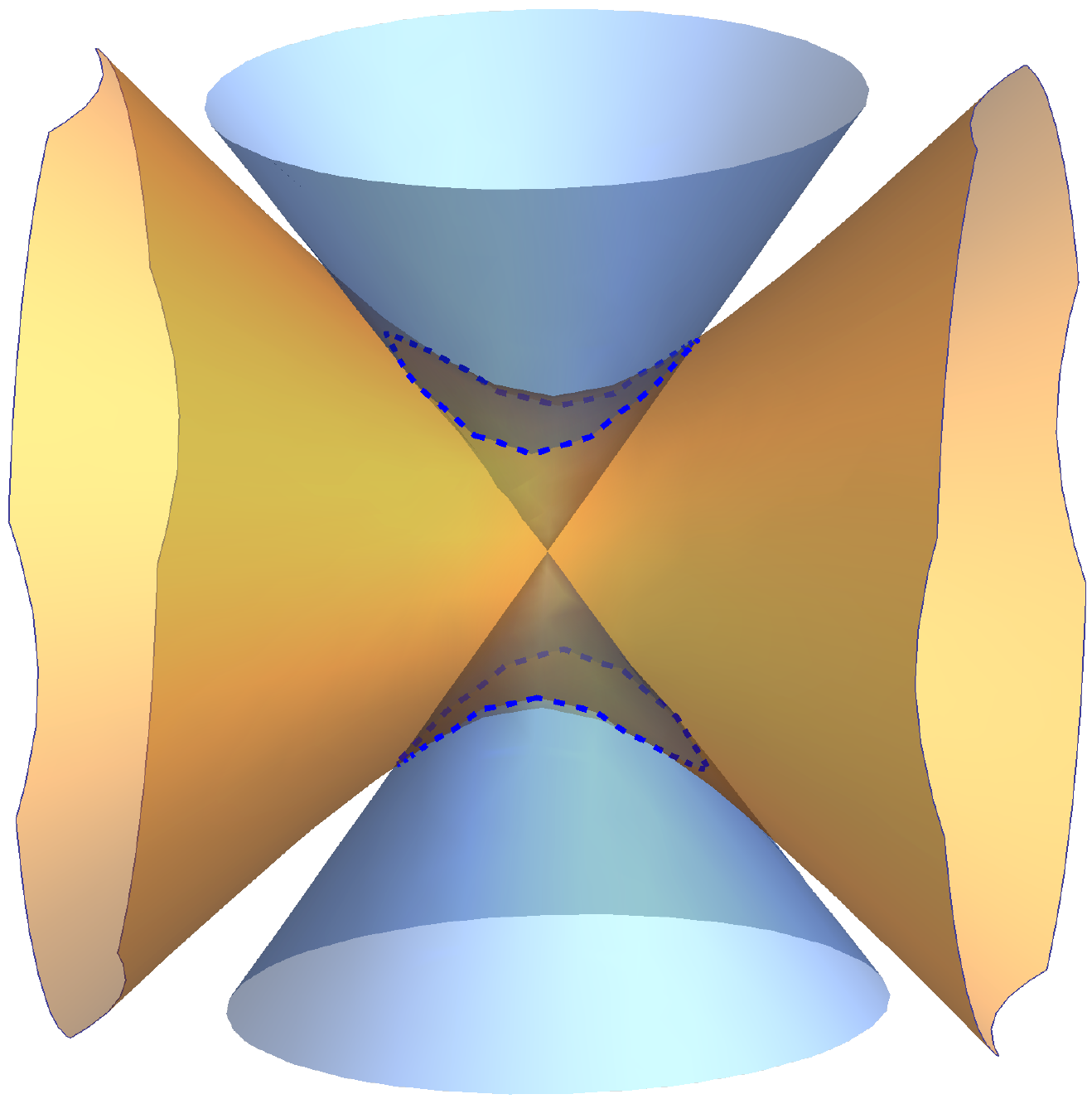}
\end{tabular}
\caption{Two geometric possibilities for the intersection of the cone $\ip{A^{-1}x}{x}=0$ and $\ip{x}{x}=1$ determining a compact domain: the collared (a) and transverse (b) $\Hyp$-ellipse. }
\label{TwoCases}
\end{figure}

\begin{remark}
In the case of the transverse $\Hyp$-ellipse we choose one of the compact domains. 
Without loss of generality we can choose the domain with $x_2>0$.
\label{Choice} 
\end{remark}

In the Klein coordinates $\xi_i = x_i/ x_0$, $i=1,2$, define the central projection by $$\pi_{\xi}: \Mink \to \R{2}, \qquad (x_0,x_1,x_2) \mapsto \left( \frac{x_1}{x_0},\frac{x_2}{x_0}\right)=:(\xi_1,\xi_2).$$ Geodesics in the Minkowski metric $ds^2$ restricted to $\Hyp$ are projected to lines in the Klein $\xi_1\xi_2$-plane \cite{Cal}. 

If two points $x,y \in \Mink$ lie on the same line $\ell$ through the origin, then $$\pi_\xi(x) = \pi_{\xi}(y)$$ as both will be scalar multiples of the direction vector of $\ell$. In particular, $\pi_\xi(x) = \pi_\xi(-x)$, so that the projection $\pi_\xi$ of the collared $\Hyp$-ellipse is a 2-to-1 mapping. 

The image of the boundary curves of the collared and transverse $\Hyp$-ellipse under the Klein projection has equation 
\begin{equation}
\frac{\xi_1^2}{b_1}+ \frac{\xi_2^2}{b_2}=1,
\label{KleinEqn}
\end{equation}
 where $1 < b_1 < b_2$ for the collared $\Hyp$-ellipse and $b_1 < 0 < 1 < b_2$ for the transverse $\Hyp$-ellipse. These curves are plane ellipses and hyperbolas, respectively. 

\begin{defn}
The \emph{pencil} of the cone $\ip{Ax}{x}=0$ in $\Mink$ is the family of cones of the form 
\begin{equation}
\ip{Ax}{x} - \lambda \ip{x}{x} = \ip{(A-\lambda I)x}{x}=0.
\end{equation}
The \emph{confocal family} consists of the dual cone $\ip{A^{-1}x}{x}=0$ and the corresponding dual cones 
\begin{equation}
\ip{(A-\lambda I)^{-1}x}{x}=0.
\label{ConfFam} 
\end{equation}
\end{defn}

In \cite{V}, the intersection of confocal quadrics with the sphere and one sheet of the hyperboloid of two sheets are studied in detail in $(n+1)$-dimensional Euclidean and Minkowski space, respectively. In particular, a factorization method from \cite{MV} is used for the integration of the billiard problem in the domain on the sphere and hyperboloid bounded by confocal families of quadrics. Further, the dynamics of such billiard systems are described in terms of hyperelliptic curves and $\theta$-functions. 

\begin{defn}
Denote by $\mathcal{C}_\lambda$ the curve of intersection of $\ip{(A-\lambda I)^{-1}x}{x}=0$ and the hyperboloid of one sheet $\Hyp$.
The curves of intersection can be bounded or unbounded, which we call \emph{elliptic-type} or \emph{hyperbolic-type}, respectively. 
\end{defn}

The curves in Figure \ref{IntersectionTypes} and the curves projected into the $x_1=0$ plane in Figure \ref{ProjOntoXZ} illustrate the previous definition. In particular, the collared and transverse $\Hyp$-ellipse each correspond to the curve $\mathcal{C}_0$. The next proposition follows directly from the definition and a direct calculation. 

\begin{prop} 
\;
\begin{enumerate}[(i)]
    \item If $0<a_0 < a_1 < a_2$, the curve $\mathcal{C}_{\lambda}$ will be of elliptic-type for $\lambda \leq a_0$ and $\mathcal{C}_{\lambda}$ will be hyperbolic-type for $a_1 \leq \lambda \leq a_2$.
    \item If $a_1 < 0 < a_0 < a_2$, the curve $\mathcal{C}_\lambda$ will be of elliptic-type for $a_1 < \lambda < a_2$ and $\mathcal{C}_\lambda$ will be hyperbolic-type for $\lambda \leq a_1$ or  $\lambda \geq a_2$.
    \item For two curves $\mathcal{C}_{\lambda_1}$ and $\mathcal{C}_{\lambda_2}$ to intersect they must be of opposite types in the case $0<a_0 < a_1 < a_2$ and must be of the same type in the case $a_1 < 0 < a_0 < a_2$.
\end{enumerate}
\label{ConfocalCurves}
\end{prop}

\begin{figure}[hbtp]
\begin{tabular}{c c}
a) \includegraphics[width=0.4\textwidth]{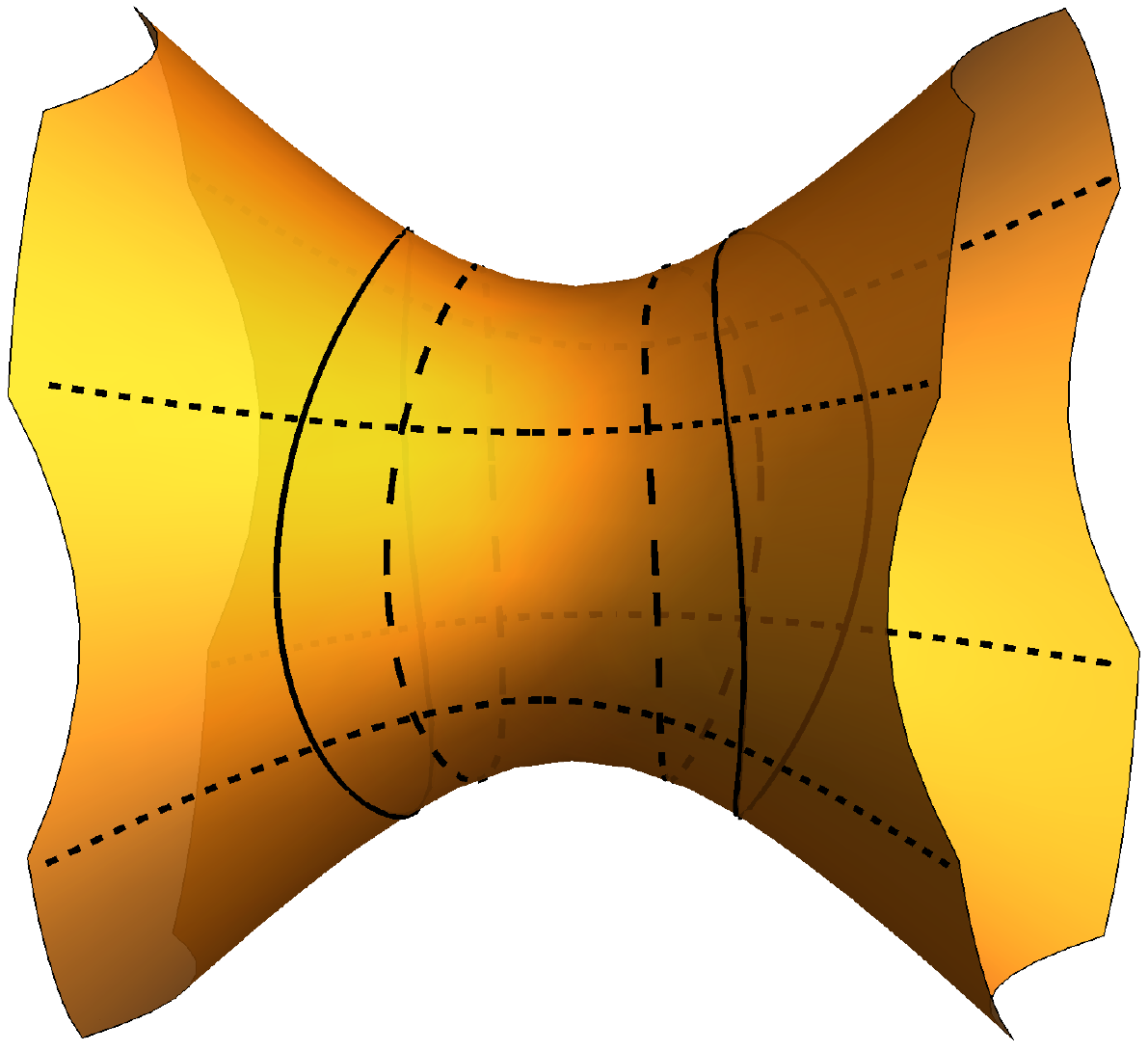}  & b) \includegraphics[width=0.50\textwidth]{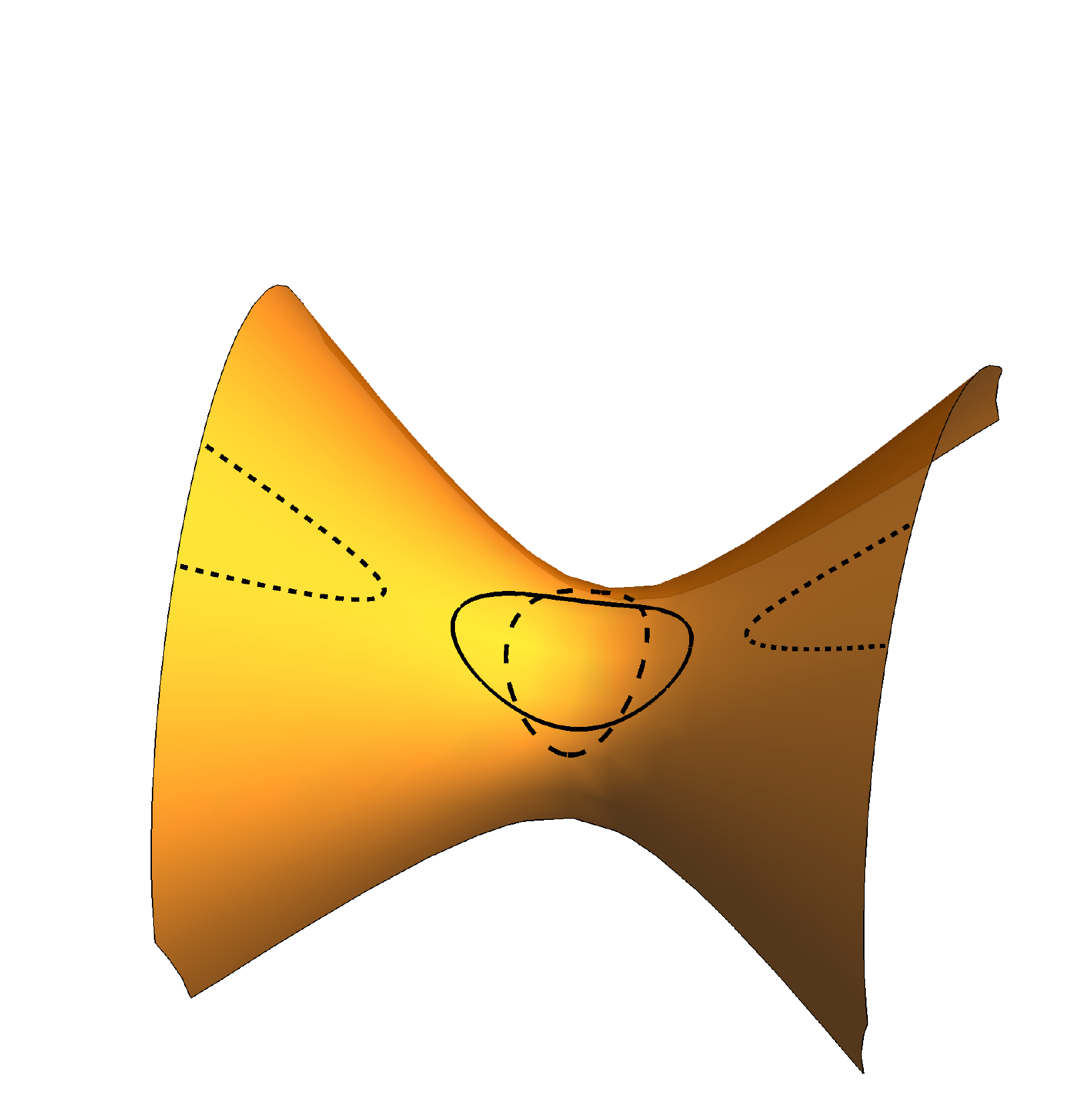}
\end{tabular}
\caption{Intersections of the confocal family with the hyperboloid of one sheet producing curves $\mathcal{C}_\lambda$ of elliptic-type (dashed) and hyperbolic-type (dotted), and the collared and transverse $\Hyp$-ellipse (solid).}
\label{IntersectionTypes}
\end{figure}

The confocal family (\ref{ConfFam}) can be written in the form 
\begin{equation}
-\frac{x_0^2}{a_0-\lambda} + \frac{x_1^2}{a_1-\lambda} + \frac{x_2^2}{a_2-\lambda}=0.
\label{ConfFam1} 
\end{equation}
For each point $(x_0,x_1,x_2) \in \Hyp$, the equation \eqref{ConfFam1} has solutions in $\lambda$ which we call the \emph{generalized Jacobi coordinates} of the point $(x_0,x_1,x_2)$.

\begin{prop}\label{NumRoots}
Let $x \in \Hyp$ and $x_0x_1x_2 \neq 0$. 
\begin{enumerate}[(i)]
    \item If $0<a_0 < a_1 < a_2$, then the equation (\ref{ConfFam1}) has two real roots in $\lambda$ satisfying the inequalities $$\lambda_1 < a_0 < a_1 < \lambda_2 < a_2.$$ 
    \item If $a_1 < 0 < a_0 < a_2$, then the equation (\ref{ConfFam1}) will have $0$ or $2$ real roots depending upon the choice of $x \in \Hyp$. If $x$ is on a ruling of $\Hyp$ connecting the foci $F_{\pm\pm}^i$, $i=0,1,2$ (see, e.g. equation (\ref{FociCoords1})), then equation (\ref{ConfFam1}) has one repeated root $\lambda$ and $\lambda \in \R{}\setminus \{a_0,a_1,a_2\}$. If there are two distinct roots then $\lambda_1$ and $\lambda_2$ satisfy either
        \begin{align}
            & \lambda_1 < \lambda_2 < a_1 < a_0 < a_2 \label{2a}\\
            & a_1 < \lambda_1 < \lambda_2 < a_0 < a_2 \label{2b}\\
            & a_1 < a_0 < \lambda_1 < \lambda_2 < a_2 \label{2c}\\
            & a_1 < a_0< a_2 < \lambda_1 < \lambda_2. \label{2d}
        \end{align}
\end{enumerate}
If $x_0x_1x_2=0$ then:
\begin{enumerate}[(i)]
\setcounter{enumi}{2}
    \item For both cases $0<a_0<a_1<a_2$ and $a_1 < 0 < a_0 < a_2$: if $x_0=0$, then $\lambda_1 = a_0$ and $a_1 \leq \lambda_2 \leq a_2$;
    \item For both cases $0<a_0<a_1<a_2$ and $a_1 < 0 < a_0 < a_2$: if $x_1=0$, then $\lambda_1 = a_1$ and $\lambda_2 \leq a_0$;
    \item If $x_2=0$, then $\lambda_1=a_2$ and 
        $\begin{cases} 
            \lambda_2 \leq a_0 &\text{ if } 0 <a_0 < a_1 < a_2 \\  
            \lambda_2 \geq a_0 &\text{ if } a_1 < 0 < a_0 < a_2. 
        \end{cases}$
\end{enumerate}
\end{prop}

\begin{proof}
Let $L(\lambda)$ denote the left hand side of (\ref{ConfFam1}). Combining terms together, the numerator is monic quadratic in $\lambda$ and has simple poles at $\lambda = a_0,a_1,a_2$. 

Consider the two cases $(i)$ and $(ii)$ separately. If $0<a_0 < a_1 < a_2$, then the discriminant of the numerator of $L(\lambda)$ can be written as 
\begin{align}
\Delta &= \left(-(a_2 - a_1) x_0^2 + (a_2 - a_0) x_1^2\right)^2 + (a_1 - a_0)^2 x_2^4  \nonumber\\
& \qquad + 2 (a_1 - a_0)\left((a_2 - a_1) x_0^2 + (a_2 - a_0) x_1^2\right) x_2^2. 
\end{align} 
Each term of the form $(a_i-a_j)$ is positive, so $\Delta>0$ and there will be two distinct roots in $\lambda$.  At the pole $\lambda=a_0$, the sign of $L(\lambda)$ changes from negative to positive while at the poles $\lambda = a_1, a_2$ the sign changes from positive to negative, which shows there is exactly one root $\lambda_2$ satisfying $a_1 < \lambda_2 < a_2$. Further, this implies there cannot be a root between $a_0$ and $a_1$. %, as such a root would be of even order to satisfy the limiting behavior between the poles $a_0$ and $a_1$ \textcolor{red}{(this explanation can be removed, but it's there to make things completely clear).} 
We can write the roots as 
 \begin{align}
 \lambda_1,\lambda_2 = \frac{\tau \pm \sqrt{\Delta}}{2}
 \end{align}
where $\tau = -(a_1 + a_2) x_0^2 + (a_0 + a_2) x_1^2 + (a_0 + a_1) x_2^2$. The roots $\lambda_1, \lambda_2$ are symmetric about $\lambda = \tau/2$ where $\tau/2 < (a_0+a_2)/2<a_2$, and so it must be the case that $\lambda_2 = (\tau +\sqrt{\Delta})/2$. Because $\lambda_1 < \lambda_2$, the only remaining possibility is that $\lambda_1 < a_0$. This proves part $(i)$. 

Next, suppose $a_1 < 0 < a_0 < a_2$. Because we assume $x_2>0$ (see remark \ref{Choice}), the coordinates $x_0,x_1$ uniquely determine $x_2$. The discriminant of the numerator of $L(\lambda)$ can  be written as 
\begin{align}
\Delta &= \left[ (a_0 - a_1) - \left((a_2 - a_0 ) x_0^2 + (a_2 - a_1 ) x_1^2\right) \right]^2 -  4 (a_2 - a_0) (a_2 - a_1) x_0^2 x_1^2,
\end{align} 
so that terms of the form $(a_i-a_j)$ are positive and $\Delta$ can be negative, zero, or positive.  

At the pole $\lambda = a_0$, $L(\lambda)$ changes from negative to positive and at the poles $\lambda = a_1, a_2$ $L(\lambda)$ changes from positive to negative. If $\Delta>0$, then the existence of one root between consecutive poles implies the second root must also be in between the same consecutive poles. This proves cases (\ref{2c}), (\ref{2d}). The cases of (\ref{2a}) and (\ref{2d}) follow from a similar argument to that of part $(i)$, as one can show that if $\lambda_2 < a_1$ then $\tau/2 < \lambda_2$, and if $a_2 < \lambda_1$ then $\lambda _1 < \tau/2$, respectively.

\begin{figure}[hbtp]
\includegraphics[width=0.9\textwidth]{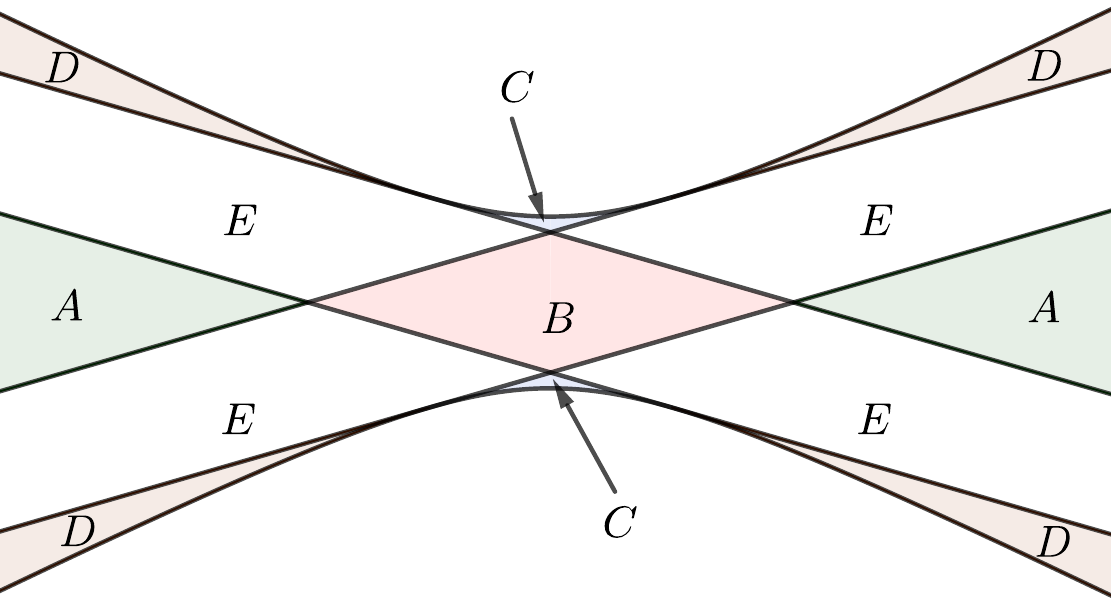}%more narrow/rectangular picture
\caption{The $x_0x_1$-plane divided into five regions when $a_1<0 < a_0 < a_2$ and $-x_0^2+x_1^2 \leq 1$. For $(x_0,x_1)$ in region $A$, (\ref{2a}) applies; region $B$, (\ref{2a}) applies; region $C$, (\ref{2c}) applies; and region $D$, (\ref{2d}) applies. There are no roots if $(x_0,x_1)$ is in region $E$ and one root if $(x_0,x_1)$ is on one of the four dividing lines. }
\label{MinkPlane}
\end{figure}

The locus of points in the $x_0x_1$-plane satisfying $\Delta=0$ are two pairs of parallel lines (see figure \ref{MinkPlane}) %that are of the form $$ \pm x_0\sqrt{a_2-a_0} \pm x_1\sqrt{a_2-a_1} = \sqrt{a_0-a_1}$$ or alternately as 
given by $$\left|x_0\sqrt{a_2-a_0} \pm x_1 \sqrt{a_2-a_1}\right| = \sqrt{a_0-a_1}.$$ These lines are tangent to the hyperbola $-x_0^2+x_1^2=1$ at the foci $F_{\pm\pm}^2$ and are the projection of rulings of $\Hyp$ connecting the foci $F_{\pm\pm}^i$ (see equation (\ref{FociCoords1}) in example \ref{confocalexample}). The four lines and this hyperbola divide the region $-x_0^2+x_1^2 \leq 1$ into five regions of interest, depicted in figure \ref{MinkPlane}. By symmetry we can address only the first quadrant and say that cases (\ref{2a}), (\ref{2b}), (\ref{2c}), (\ref{2d}), correspond to $(x_0,x_1)$ in regions $A$, $B$, $C$, $D$, respectively. %green, red, blue, brown regions, respectively.  
If $x$ is in region $E$ between the lines then $\Delta<0$ and $L(\lambda)$ has no real roots. 

Lastly, the cases when $x_0x_1x_2=0$ follow from a direct calculation.
\end{proof}

This proposition is analogous to Proposition 2 of \cite{V} and Theorem 4.5 of \cite{KT}. However, due to our definition of the confocal family, we do not have the same topological considerations as stated in \cite{KT}. 

Because the collared and transverse $\Hyp$-ellipse are precisely the curves of intersection with $\Hyp$ corresponding to the member of the confocal family (\ref{ConfFam1}) with $\lambda=0$, we can reduce the previous propositions to a simpler statement in terms of generalized Jacobi coordinates. 

\begin{cor}
For any $x$ in the interior of the collared $\Hyp$-ellipse, the generalized Jacobi coordinates of $x$ satisfy $0 < \lambda_1 \leq a_0$, $a_1 \leq \lambda_2 \leq a_2$. And for any $x$ in the interior of the transverse $\Hyp$-ellipse, the generalized Jacobi coordinates of $x$ satisfy $a_1 \leq \lambda_1 < 0 <\lambda_2 \leq a_0$. 
\end{cor}

\noindent From the previous propositions we conclude the following theorem.

\begin{theorem} 
For any point $x$ on $\Hyp$ there are either 0, 1, or 2 confocal conics of the form (\ref{ConfFam1}) passing through $x$. Further, in the case of two confocal conics, the confocal conics are pairwise orthogonal with respect to the Minkowski inner product $\ip{\cdot}{\cdot}$. 
\end{theorem}

\noindent The orthogonality proof follows from the same argument given in Theorem 4.5 in \cite{KT}.

\begin{example}\label{confocalexample}
To illustrate properties of the confocal family on $\Hyp$, consider the confocal family
\begin{equation}
\ip{(A-\lambda I)^{-1}x}{x}=0
\end{equation}
in $\Mink$. The initial cone with $\lambda=0$ is given by the equation $$-\frac{x_0^2}{a_0} + \frac{x_1^2}{a_1} + \frac{x_2^2}{a_2}=0.$$ The foci corresponding to the degenerate case $\lambda = a_1$ of the confocal family 
$$-\frac{x_0^2}{a_0-\lambda} + \frac{x_1^2}{a_1-\lambda} + \frac{x_2^2}{a_2-\lambda}=0$$
have coordinates 
\begin{equation}
F_{\pm \pm}^1=\left( \pm \sqrt{\frac{a_0-a_1}{a_2-a_0}}, 0, \pm \sqrt{\frac{a_2-a_1}{a_2-a_0}}\right).
\label{FociCoords1}
\end{equation}
% This expression can be found by considering the intersection of the surfaces $$ -\frac{x_0^2}{a_0 -\lambda} + \frac{x_1^2}{a_1-\lambda} + \frac{x_2^2}{a_2-\lambda}=0 \;\;\; \text{ and }\;\;\; -x_0^2 + x_1^2 + x_2^2=1,$$ setting $x_1=0$, then setting $\lambda = a_1$ and solving. 
The other sets of degenerate foci corresponding to $\lambda = a_0$ and $\lambda=a_2$, $F_{\pm\pm}^0$ and $F_{\pm\pm}^2$, respectively, can be calculated similarly. 
%\begin{equation}
%F_{\pm \pm}^0=\left( 0, \pm \sqrt{\frac{a_0-a_1}{a_2-a_1}}, \pm \sqrt{\frac{a_2-a_0}{a_2-a_1}}\right),\; F_{\pm \pm}^2=\left( \pm \sqrt{\frac{a_2-a_0}{a_0-a_1}}, \pm \sqrt{\frac{a_2-a_1}{a_0-a_1}},0\right).
%\label{FociCoords0}
%\end{equation}
%%and repeating the calculation in the case of $\lambda = a_2$ results in 
%%\begin{equation}
%%F_{\pm \pm}^2=\left( \pm \sqrt{\frac{a_2-a_0}{a_0-a_1}}, \pm \sqrt{\frac{a_2-a_1}{a_0-a_1}},0\right).
%%\label{FociCoords2}
%%\end{equation}
When real, the coordinates of the foci are shown in figure \ref{ProjOntoXZ}. Further, the confocal family has three other degenerate conics: $\lambda = a_0$ corresponds to the $x_2$-axis; $\lambda = a_2$ corresponds to the $x_0$-axis; and the line at infinity corresponds to $\lambda = \pm \infty$.

\begin{figure}[btp]
\begin{tabular}{c c}
a) \includegraphics[width=0.44\textwidth]{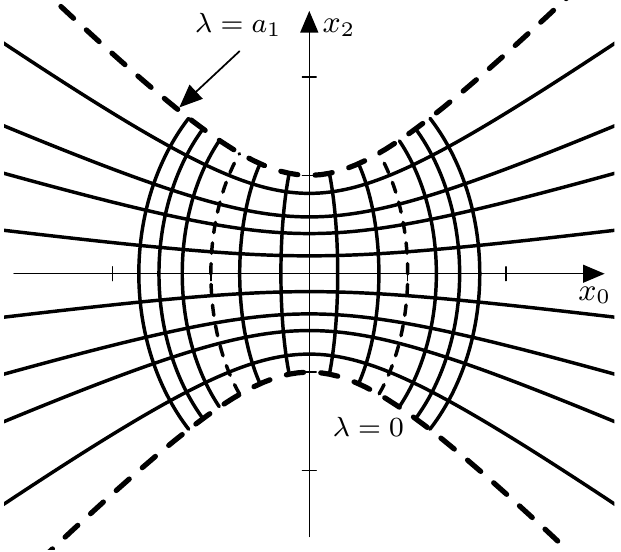}  
&
b) \includegraphics[width=0.44\textwidth]{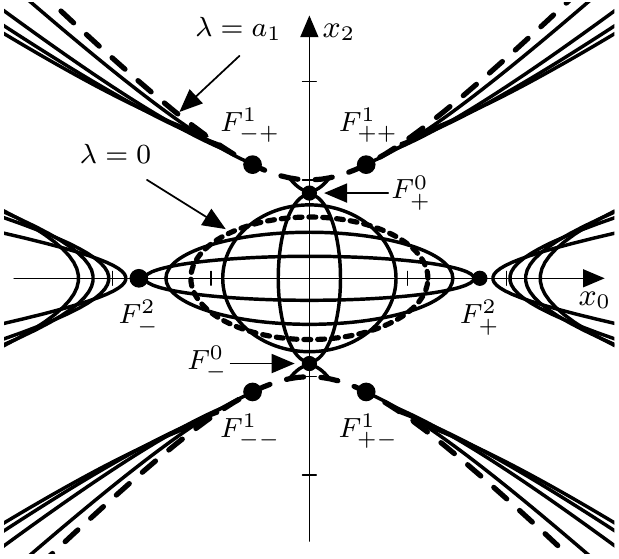}
\end{tabular}
\caption{The projection of the intersections of the confocal family (\ref{ConfFam1}) and the hyperboloid of one sheet $\ip{x}{x}=1$ onto the $x_0x_2$-plane. The cases when $0 < a_0 < a_1 < a_2$ and $a_1 < 0 < a_0 < a_2$ are shown in (a) and (b), respectively. The foci $F_{\pm}^0$ and $F_{\pm}^2$ shown are the projections of $F_{\pm\pm}^0$ and $F_{\pm \pm}^2$, respectively, onto the $x_0x_2$-plane. }%These curves are given by (\ref{ProjectedCurvesEqn}).}
\label{ProjOntoXZ}
\end{figure}

After this projection onto the $x_1=0$ plane, the family of curves resembles that of confocal conics in the Minkowski plane $\mathbf{M}^2$. See \cite{BM,DR3,DR4} for a review of the basic properties. In such a setting, these curves are of the form 
\begin{equation}
\frac{x_0^2}{\frac{a_0-\lambda }{a_1-a_0}}-\frac{x_2^2}{\frac{a_2-\lambda }{a_1-a_2}}=1
%\frac{y^2}{\frac{b-\lambda }{b-c}}-\frac{x^2}{\frac{a-\lambda }{a-c}}=1 is the corresponding solution with y's and x's
\label{ProjectedCurvesEqn}
\end{equation}
and are shown in figure \ref{ProjOntoXZ} for varying $\lambda$. %\textcolor{red}{The curves in figure \ref{ProjOntoXZ}a will be arcs of ellipses for $\lambda <a_0$ and arcs of hyperbolae for $a_1 < \lambda < a_2$, while the curves in figure \ref{ProjOntoXZ}b will be ellipses for $a_0 < \lambda < a_2$ and arcs of hyperbolae for $\lambda \in (-\infty,a_1) \cup (a_1,a_0) \cup (a_2,\infty).$ }
%Figure \ref{ProjOntoXZ} shows the case when $a_1 < 0 < a_0 < a_1$ while figure \ref{ProjOntoXZ2} shows the case when $0 < a_0 < a_1 < a_2$. %These confocal curves (\ref{ProjectedCurvesEqn}) have four foci in the $x_0x_2$-plane with coordinates 
%\begin{equation*}
%\left(0,\pm \sqrt{\frac{a_0-a_2}{a_1-a_2}}\right), \;\;\; \left(\pm \sqrt{\frac{a_0-a_2}{a_1-a_0}},0 \right).
%\end{equation*}
The equations of the four separating lines are easily derived from the four degenerate foci. 
\end{example}

%\section{The Geometry of $\Mink$ and Geodesics on $\Hyp$}
\section{Billiards on $\Hyp$}
\label{BonH}

The study of billiards in pseudo-Euclidean spaces was first discussed in \cite{KT}, while billiards within confocal families in pseudo-Euclidean spaces is discussed in \cite{DR3,DR4}. 

To begin our examination of billiards on $\Hyp$, we interpret the billiard flow in the collared or transverse $\Hyp$-ellipse as consisting of tangent vectors which coincide with the geodesic flow on $\Hyp$. Suppose a vector $v$ hits the boundary at a point $p$ and let $n_p$ be the normal vector of $T_p\Hyp$, the tangent plane to $\Hyp$ at $p$. The billiard motion stops if $p$ is a singular point, i.e. $\ip{n_p}{n_p}=0$. If $p$ is not a singular point, then $n_p$ is transverse to $T_p\Hyp$. Decompose $v = t + n$ into its normal and tangential components so that its reflection is $v^\prime = t-n$. Clearly $|v|^2 = |v^\prime|^2$, so the type of the geodesic is preserved under the billiard reflection. In particular, we note that by construction the boundaries of the collared and transverse $\Hyp$-ellipses consist entirely of nonsingular points. 

Armed with specific criteria for geodesic connectedness on $\Hyp$ from proposition \ref{GeodesicProp}, we apply these concepts to billiards inside the collared and transverse $\Hyp$-ellipse. 

\begin{theorem}
Let $p$ and $q$ be two points on the hyperboloid $\Hyp$.
\begin{enumerate}[(i)]
    \item  If $p$ and $q$ are nonantipodal points on opposing component curves of the collared $\Hyp$-ellipse, then there exists a unique geodesic connecting $p$ to $q$. If the geodesic is light- or time-like, the arc of the geodesic from $p$ to $q$ is contained entirely inside the collared $\Hyp$-ellipse. If the geodesic is space-like, then the distance-minimizing arc of the geodesic is contained in the collared $\Hyp$-ellipse.  
    \item Let $p$ and $q$ be distinct points on the same component curve of the collared $\Hyp$-ellipse. Then $p$ and $q$ can either be connected by a space-like geodesic or $p$ and $q$ cannot be connected by a geodesic.  The length-minimizing arc of the space-like geodesic connecting $p$ and $q$ lies outside the collared $\Hyp$-ellipse. 
    \item Let $p$ and $q$ be distinct points on the transverse $\Hyp$-ellipse. Then $p$ and $q$ can be connected by an arc of a geodesic which stays entirely within the transverse $\Hyp$-ellipse.
\end{enumerate}
\label{DifferentBoundary}
\end{theorem}

\begin{proof}
Fix a point $p$ on one component curve, $\mathcal{E}_+$, of the collared $\Hyp$-ellipse. The plane $\ip{p}{x}=1$ intersects the collared $\Hyp$-ellipse in three distinct points, $p$, $q$, and $r$, where $q$ and $r$ are on the opposite component curve, $\mathcal{E}_-$, of the collared $\Hyp$-ellipse. The plane $\ip{p}{x}=1$ separates $\mathcal{E}_-$ into two disjoint curves, one whose points can be connected to $p$ by a time-like geodesic, and the other whose points can be connected to $p$ by a space-like geodesic. In the case of the space-like geodesics, the length-minimizing arc of the space-like geodesic will lie between $\mathcal{E}_-$ and $\mathcal{E}_+$ and represents the billiard trajectory. By construction, the points $q$ and $r$ are the only points on $\mathcal{E}_-$ that can be connected to $p$ with a light-like geodesic. Further, every point $s$ on $\mathcal{E}_-$ satisfies $\ip{p}{s} \geq -1$, with equality holding if and only if $s=-p$. 

Next, consider which points on $\mathcal{E}_+$ can be connected to $p$ by a geodesic. The plane $\ip{-p}{x}=1$ intersects the collared $\Hyp$-ellipse in three points, namely $-p$, $-q$, and $-r$, where now $-q$ and $-r$ are on $\mathcal{E}_+$. The plane $\ip{-p}{x}=1$ divides $\mathcal{E}_+$ into two disjoint curves, one whose points can be connected to $p$ by a space-like geodesic, and the other whose points cannot be connected to $p$ by a geodesic. In the case of the space-like geodesic, the length-minimizing arc that connects the point to $p$ will lie outside $\mathcal{E}_+$ and represents the billiard trajectory. All points $s$ on $\mathcal{E}_+$ satisfy $\ip{p}{s} \leq 1$ with equality holding if and only if $s=p$. 

Lastly, suppose $p$ is a point on the transverse $\Hyp$-ellipse. The plane $\ip{p}{x}=1$ intersects the transverse $\Hyp$-ellipse in three points $p$, $q$ and $r$. Unlike the case of the collared $\Hyp$-ellipse, these points are not necessarily distinct. If $p$ is a point on the transverse $\Hyp$-ellipse that has a light-like tangent, then exactly one of the points $q$ or $r$ coincides with $p$. The plane $\ip{p}{x}=1$ separates the transverse $\Hyp$-ellipse into two disjoint curves, one of whose points can be connected to $p$ by a time-like geodesic and the other whose points can be connected to $p$ by a space-like geodesic. In the case of the space-like geodesic, the length-minimizing arc of the geodesic represents the billiard trajectory and lies entirely inside the transverse $\Hyp$-ellipse. All points $s$ on the transverse $\Hyp$-ellipse satisfy $\ip{p}{s} > -1$.
\end{proof}

\begin{figure}[hbtp]
\begin{tabular}{c c}
a) \includegraphics[width=0.45\textwidth]{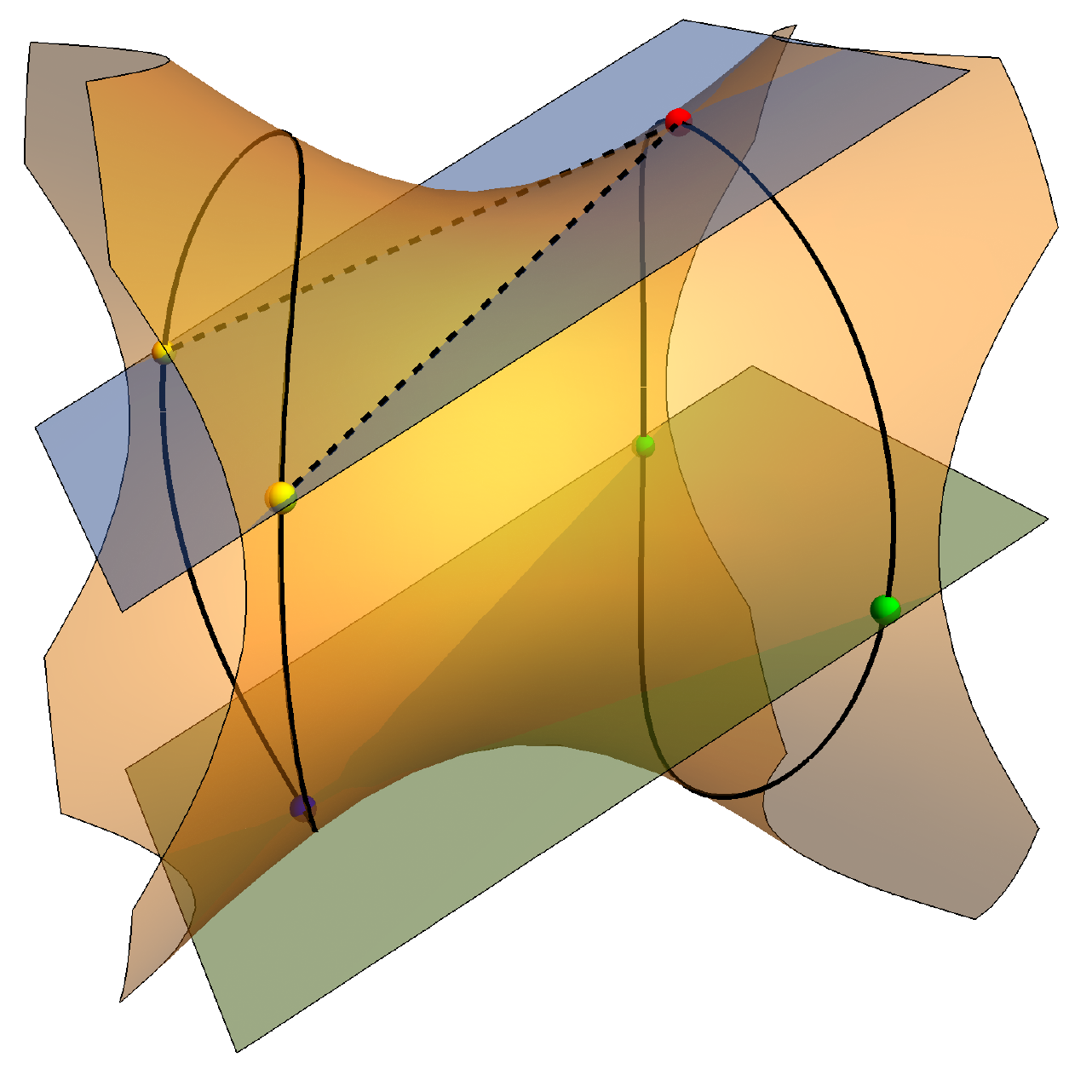} & b) \includegraphics[width=0.45\textwidth]{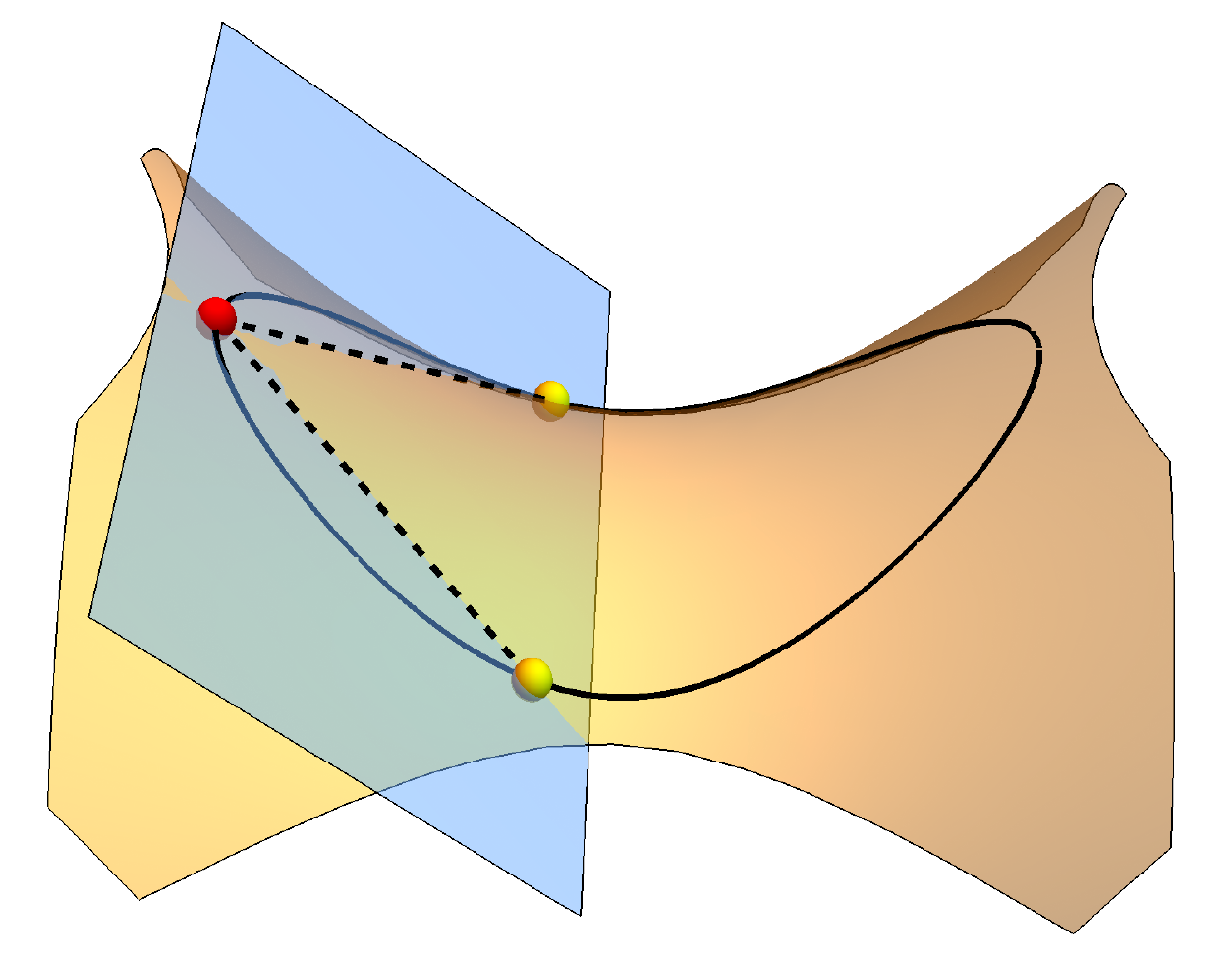} \\
\end{tabular}
\caption{The point $p$ connected to the points $q$, $r$ by light-like geodesics (dashed lines) in the plane $\ip{p}{x}=1$ in the case of the collared (a) and transverse (b) $\Hyp$-ellipse. Figure (a) also shows the plane $\ip{-p}{x}=1$ and the antipodal points $-p$ on $\mathcal{E}_-$ and $-q$, $-r$ on $\mathcal{E}_+$.}
\label{SeparatingPlanes}
\end{figure}

As noted in the above proof, the existence of odd billiard trajectories on the hyperboloid of one sheet will depend upon which restrictions are placed upon the billiard and whether the billiard is inside or outside the collared or transverse $\Hyp$-ellipse. 

\begin{cor} 
 \;
\begin{enumerate}[(i)]
    \item Suppose billiard trajectories are required to stay within the collared $\Hyp$-ellipse. Then any periodic billiard trajectory in the collared $\Hyp$-ellipse must have even period. These trajectories can be space-, light-, or time-like. 
    \item Suppose billiard trajectories are allowed outside the collared $\Hyp$-ellipse but are still restricted to motion on $\Hyp$. Then all such trajectories are space-like and reflect off of one component curve of the collared $\Hyp$-ellipse. Further, periodic trajectories can have even or odd period.  
    \item Billiard trajectories inside the transverse $\Hyp$-ellipse can be space-, light-, or time-like. Space- and time-like trajectories can have either even or odd period, while closed light-like trajectories are always even-periodic. 
\end{enumerate}
\label{DifferentBoundary2}
\end{cor}

We also note that the generatrices of the hyperboloid of one sheet are all light-like. See section 4 of \cite{DR3} or \cite{KT} for details.

\section{Factorization Method of Matrix Polynomials}\label{FMMP}

In \cite{MV} a method is proposed to determine the integrability of a discrete dynamical system by reducing the problem to the factorization of matrix polynomials. In particular, this approach is applied to a discrete version of rigid body dynamics, discrete dynamics on Stiefel manifolds, and billiards inside an ellipsoid. In \cite{V} this method is used to show the integrability of billiards on the sphere and Lobachevsky space for domains bounded by confocal quadrics. We provide a summary of the technique in $\Mink$ below, noting that the statements made below extend to the Minkowski space of arbitrary finite dimension. 

%\textcolor{red}{What appears below is a near-direct reproduction from Veselov's paper and should be adjusted/rewritten appropriately.} \\

%Recall the algorithm for finding the eigenvalues of a given matrix $A=A_1$. Factorize this matrix as $$A_1 = Q_1 R_1$$ where $Q_1$ and $R_1$ are orthogonal and upper-triangular matrices, respectively. %A theorem of linear algebra states that if $A$ is nondegenerate and $A=QR = \hat{Q}\hat{R}$ are two QR decompositions, then $\hat{Q} = QS$ and $\hat{R} = SR$ for some square diagonal matrix $S$ with entries $\pm 1$. So we may assume
%If we assume $A$ is nondegenerate and the diagonal entries of $R_1$ are positive, then this factorization is unique and given by the Gram-Schmidt orthogonalization process. 
%
%The second step begins by considering the matrix $$A_2 = R_1 Q_1$$ and its factorization $$A_2 = Q_2 R_2.$$ Iterating this procedure produces a sequence of matrices $A_1, A_2, \ldots,$ which are similar, as $R_i Q_i = Q_i^{-1}(Q_i R_i)Q_i$. Under certain assumptions, such as symmetric matrices with distinct eigenvalues, this sequence has a limit that is the diagonalized form of $A=A_1$. 

It was shown in \cite{MV} if we start from a certain quadratic matrix polynomial $L(\lambda)$ 
$$L(\lambda) = \ell_0 + \ell_1 \lambda + \ell_2 \lambda^2$$ 
and its factorization of the form 
$$L(\lambda) = (b_0 + b_1\lambda) (c_0 + c_1 \lambda) = B(\lambda)C(\lambda),$$
then the analogous procedure
$$L(\lambda) \to L^\prime(\lambda) = C(\lambda)B(\lambda) = C(\lambda)L(\lambda)C^{-1}(\lambda)$$ corresponds to dynamics of the discrete versions of some classical integrable systems, in particular, the billiard dynamics in ellipsoids in $\R{n}$. 

%Now we shall show that this method works in our setting. 
Let $x$, $y$, and $z$ be the successive reflection points in the billiard on $\Hyp$ and inside the confocal family (\ref{Eqn2}):
\begin{equation}
\ip{A^{-1}x}{x} = \ip{A^{-1}y}{y} = \ip{A^{-1}z}{z} =0.
\label{InitialPoints}
\end{equation}
In particular, this means that $x$, $y$, $z$ cannot be antipodal points of each other. In the projective Klein model we have the straight lines $xy$ and $yz$, which are in one plane with the normal $N$ to the collared or transverse $\Hyp$-ellipse (\ref{KleinEqn}) and form with $N$ angles which are equal in the induced metric. See figure \ref{BillRef}.

\begin{figure}[htbp]
\includegraphics[scale=1.0]{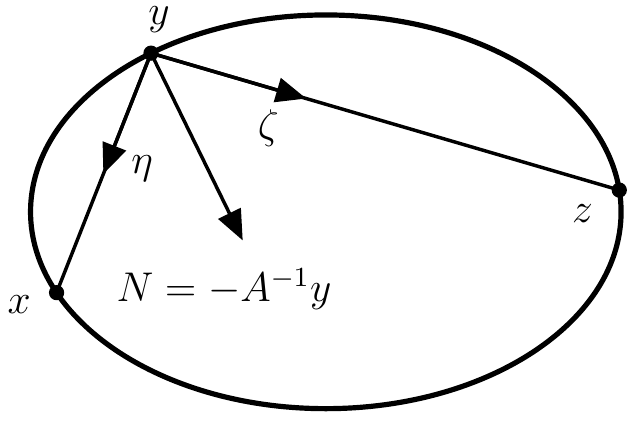}
\caption{Billiard reflection with respect to the Lorentzian metric. }
\label{BillRef}
\end{figure}

The main point of the factorization method is to find the corresponding matrix polynomial $L(\lambda)$. As is the case with the Lobachevsky space and the sphere, in our case $L(\lambda)$ is linear:
\begin{equation}
L(\lambda) = A + \lambda(x \otimes y^* - y \otimes x^*).
\label{MatPoly}
\end{equation}
where the bivector
\begin{equation*}
x\wedge y := x\otimes y^* - y \otimes x^*
\end{equation*}
is the skew-symmetric operator in $\Mink$. The norm squared of this bivector  
\begin{equation*}
\left| x \wedge y \right|^2 := \ip{x}{y}^2 - \ip{x}{x}\ip{y}{y}
\end{equation*}
is the area of the parallelogram generated by $x$ and $y$ in $\Mink$. 

After two steps the factorization procedure leads to the transformation $L(\lambda) \to L^{\prime\prime}(\lambda)$, which corresponds to the billiard dynamics $(x,y) \to (y,z)$. The arguments of \cite{V} carry over to the hyperboloid of one sheet, and the result is stated in the theorem below. 

\begin{theorem}
Let $\{x_k\}$ be an orbit in the billiard problem in the collared or transverse $\Hyp$-ellipse domain of $\Hyp$, which in the projective representation in $\Mink$ is determined by the equation $\ip{Ax}{x}\geq 0$. Choose the vectors $x_k$ in such a way that $|x_k \wedge x_{k+1}|^2$ is constant. Then the matrix
$$L_k = A + \lambda x_{k-1} \wedge x_k $$
undergoes the isospectral transformation 
\begin{equation}
L_{k+1} = A_k L_k A_k^{-1}
\label{Thm1Eqn1}
\end{equation}
where
\begin{equation}
A_k = A- \lambda (\zeta_k \otimes x_k^* + x_k \otimes \eta_k^*).
\label{Thm1Eqn2}
\end{equation}
Here $\zeta_k$ and $\eta_k$ are tangent vectors to the trajectory at the reflection point $x_k$ as shown in figure \ref{BillRef}. 
%Moreover, $L_{k+1}$ is the result of two steps of the factorization procedure described in \ci and applied to $L_k$. 
\label{Thm1}
\end{theorem}

The relations (\ref{Thm1Eqn1}) and (\ref{Thm1Eqn2}) follow from the previous considerations but can be checked also by straightforward calculation. 

\begin{cor}
The billiard in the collared and transverse $\Hyp$-ellipse has the following integrals $F_j:$
\begin{equation}
F_j = \dsum{i\neq j}{}{\dfrac{J_i J_j (x_i y_j - x_j y_i)^2}{a_j - a_i}} \;\;\; (j = 0, 1, 2)
\end{equation}
which satisfy the unique relation
$$F_0 + F_1 + F_2 = 0$$
and $-J_0=J_1 = J_2=1$ is given by the signature of the metric in $\Mink$.
\end{cor}

%Recall that for $\Mink$ we have $-J_0=J_1 = J_2=1$. The corollary follows from Theorem \ref{Thm1} and the formula
The corollary follows from Theorem \ref{Thm1} and the formula

\begin{equation} 
\det(L - \mu I) = \det(A - \mu I)(1-\lambda^2 \phi_\mu(x,y)),
\label{spec1}
\end{equation}
where 
\begin{equation}
\begin{split}
\phi_\mu(x,y) &= \ip{(A-\mu I)^{-1}x}{y}^2 - \ip{(A-\mu I)^{-1}x}{x}\ip{(A-\mu I)^{-1}y}{y} \\
 &= \sum_{i=0}^2 \dfrac{F_i}{a_i-\mu}.
 \end{split}
\label{phieqn}
\end{equation}

One can show that these integrals are in involution with respect to the natural symplectic structure. Therefore, this billiard problem is integrable in the sense of Liouville. 

\begin{figure}[htbp]
\begin{tabular}{c c}
a) \includegraphics[width=0.50\textwidth]{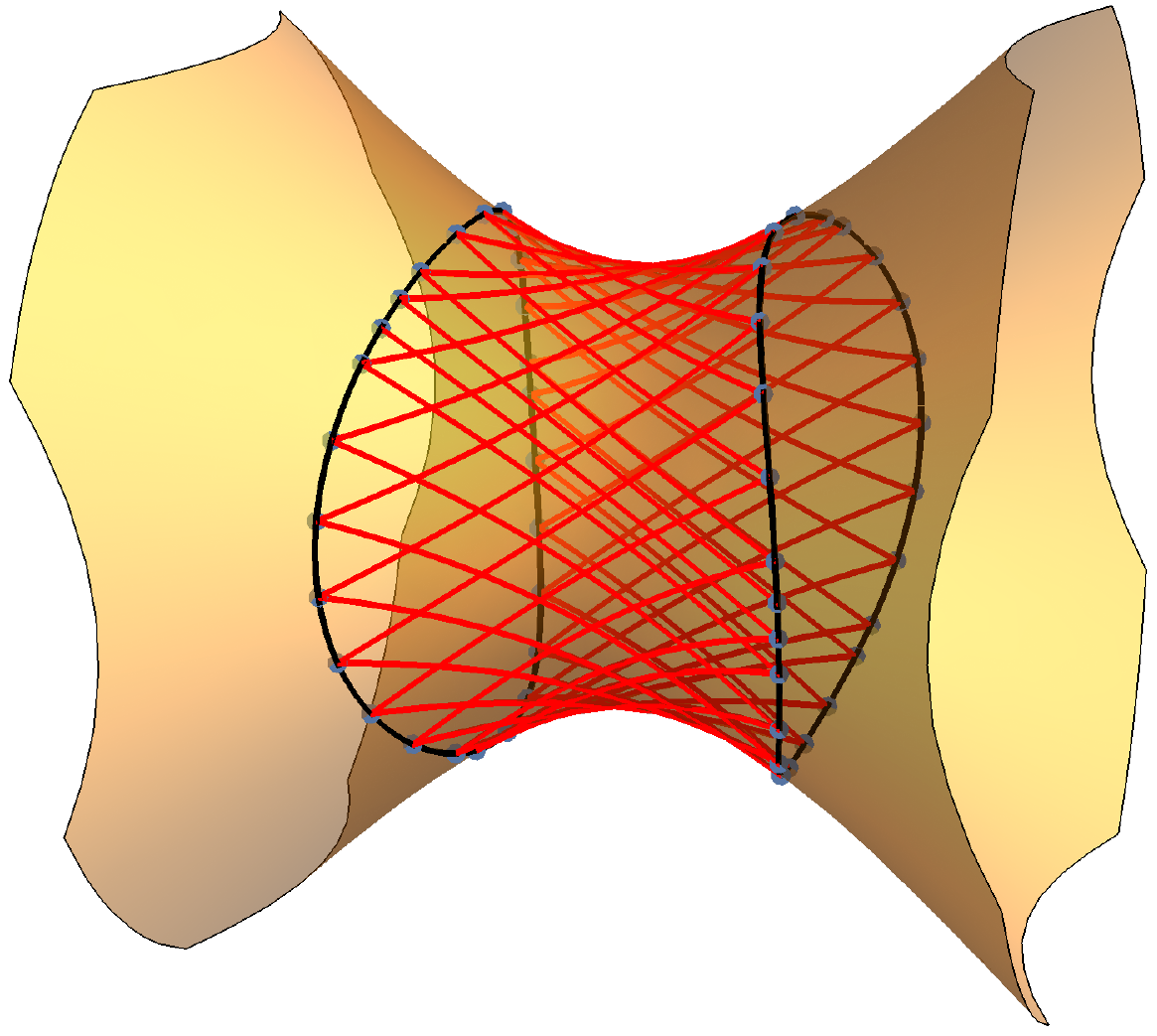} & b) \includegraphics[width=0.40\textwidth]{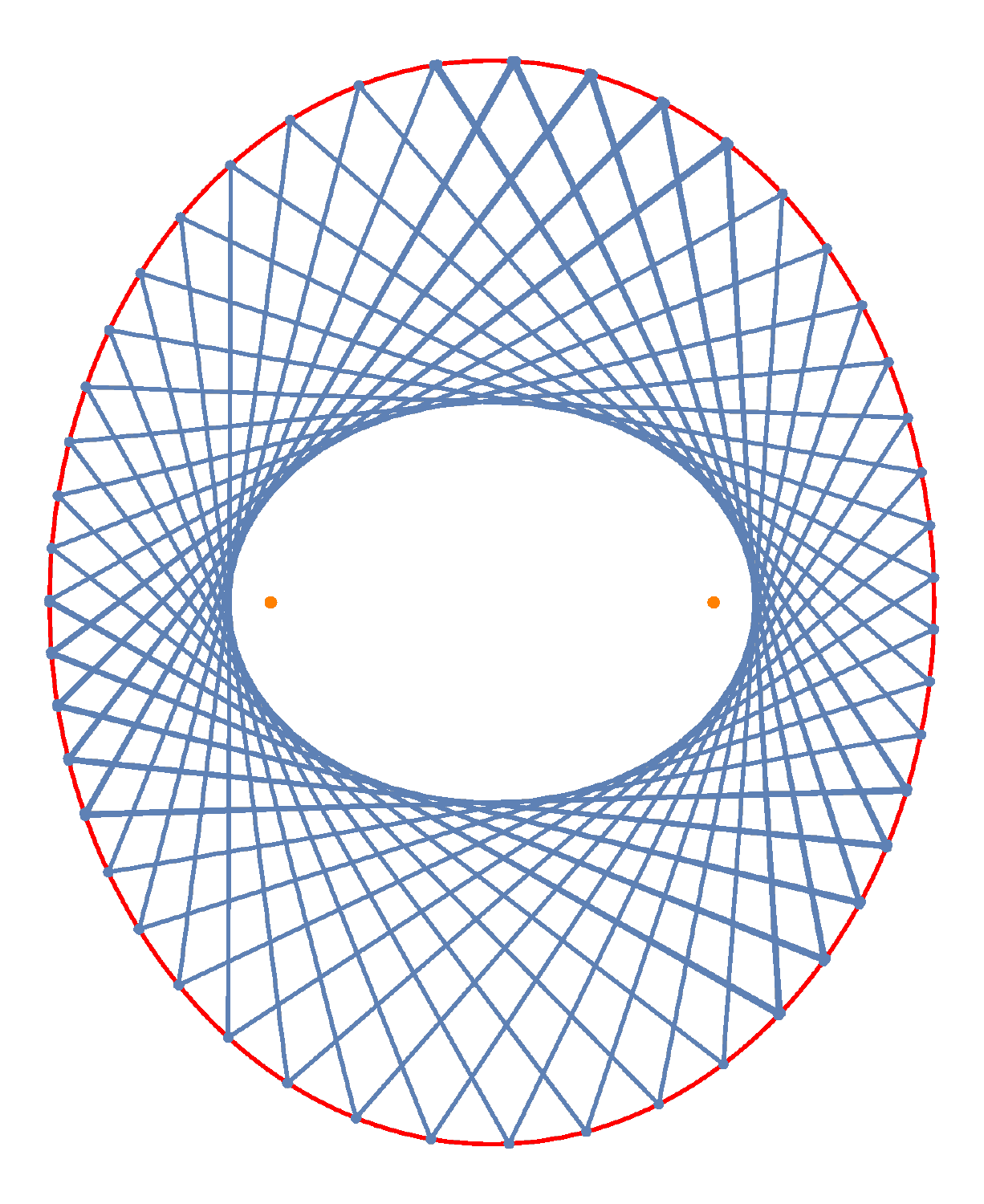} \\
c) \includegraphics[width=0.45\textwidth]{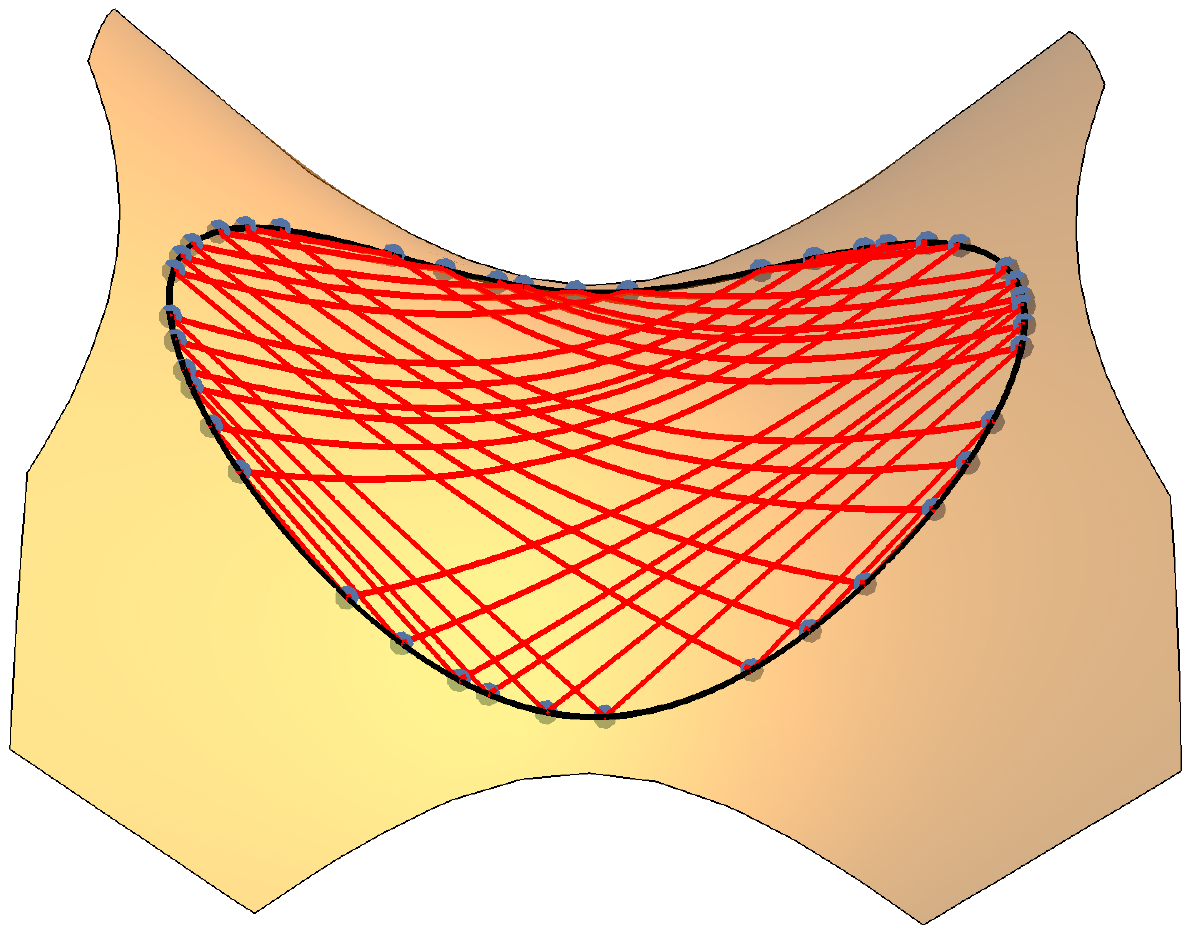} & d) \includegraphics[width=0.45\textwidth]{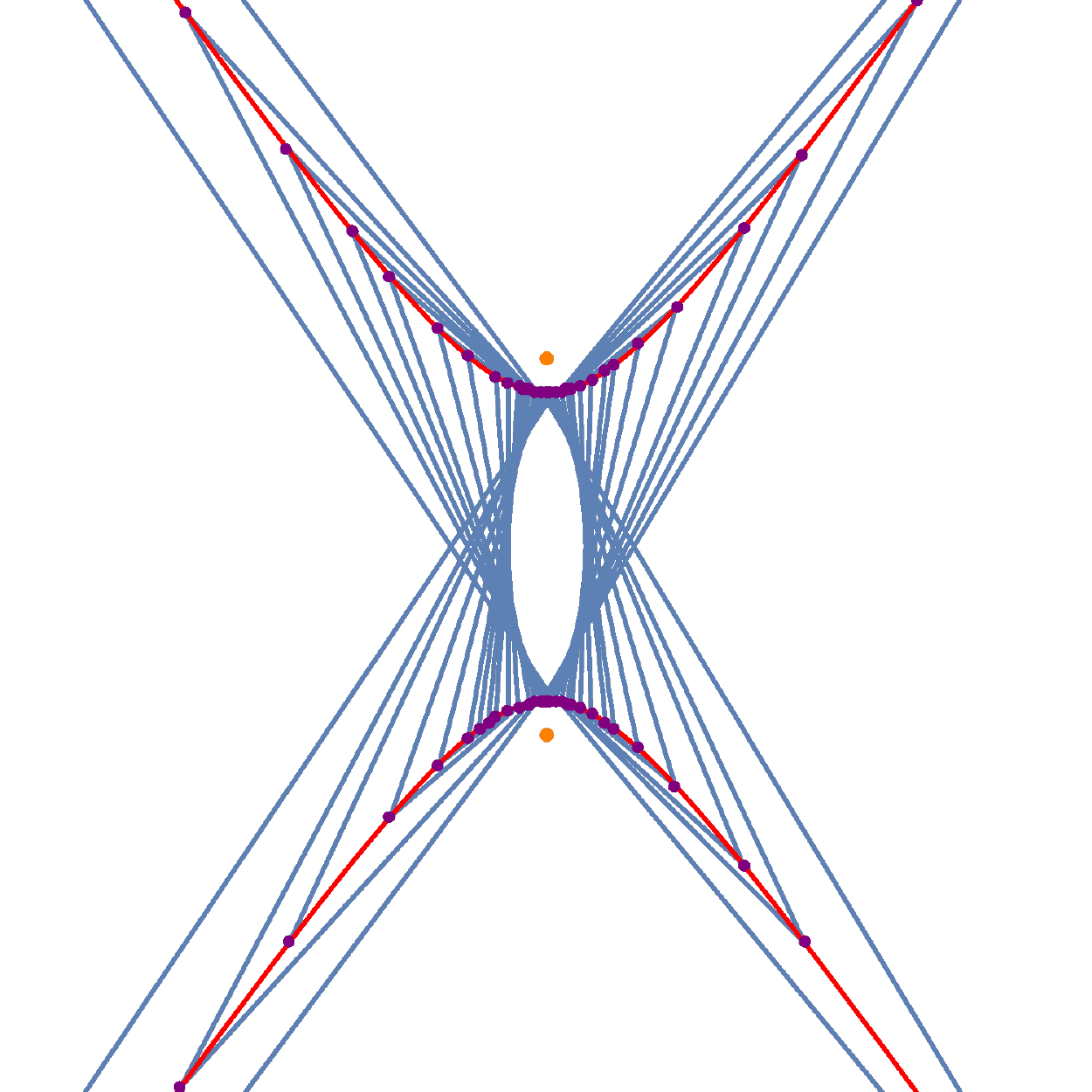} \\
\end{tabular}
\caption{Sample trajectories in the case of the collared and transverse $\Hyp$-ellipse, (a) and (c), and their projections into the Klein coordinates $\xi_1, \xi_2$, (b) and (d), respectively.}
\label{TrajEx}
\end{figure}

\begin{remark}
The matrix factorization algorithm of Veselov is blind to the discussion of geodesics from Section \ref{BonH}. At the start, the points $x$ and $y$ must satisfy equation (\ref{InitialPoints}) and be nonantipodal points on $\Hyp$. But consider two points $x$, $y$  on the same component curve $\mathcal{E}_+$ of the collared $\Hyp$-ellipse such that $x$ and $y$ cannot be connected by a geodesic on $\Hyp$ (i.e. $\ip{x}{y} \leq -1$). The algorithm produces the next collision point $z$ on $\mathcal{E}_+$ for the billiard dynamics of $(x,y) \to (y,z)$. Because the reflection law preserves the type of geodesic, the points $y$ and $z$ also cannot be connected by a geodesic on $\Hyp$ (i.e. $\ip{y}{z}\leq -1$), and hence all successive points produced from the algorithm cannot be connected by a geodesic on $\Hyp$. However, a valid trajectory whose collision points can be connected by geodesics can be recovered from this strange situation: starting with the pair $(x,-y)$ the algorithm produces the same (now valid) reflection point $z$, so the billiard dynamics are $(x,-y) \to (-y,z)$! As $\ip{x}{y} \leq -1$, it must be the case that $\ip{x}{-y} \geq 1$, and so these points can be connected by  time- or light-like geodesics. Moreover, both the invalid and valid billiard trajectories project to the same billiard in Klein coordinates because the projection $\pi_\xi$ maps antipodal points to the same point. This discussion proves the following proposition.
\end{remark}

\begin{prop}
Suppose $$(y_1,y_2) \to (y_2,y_3) \to (y_3,y_4) \to (y_4,y_5) \to \cdots$$ is a sequence of billiard reflections of the collared $\Hyp$-ellipse produced by the Veselov matrix factorization algorithm. Further suppose $\ip{y_1}{y_2}\leq -1$ so that the initial points (and hence all successive points) cannot be connected by geodesics on $\Hyp$. Then the sequence $$(y_1,-y_2) \to (-y_2,y_3) \to (y_3,-y_4) \to (-y_4,y_5) \to \cdots$$ is a sequence of billiard reflections of the collared $\Hyp$-ellipse produced by the matrix factorization algorithm, all of whose points can be connected by time- or light-like geodesics on $\Hyp$. Moreover, both billiard sequences project to the same trajectory in Klein coordinates. 
\label{InvalidValidTraj}
\end{prop}

Let $\mathcal{AA}$ be the \emph{alternating antipodal map} whose image is described in the proposition above. This map sends a sequence of billiard collisions to a sequence where every other point has been sent to its antipode:
$$\mathcal{AA}: \{ (y_k, y_{k+1})\}_{k \in \N} \mapsto \left\{\left( (-1)^{k+1} y_k, (-1)^k y_{k+1}\right)\right\}_{k \in \N}.$$
As discussed above, we will only need to consider this map in the case of the collared $\Hyp$-ellipse. Clearly $\mathcal{AA}$ is an involution on the space of sequences of billiard collisions. Proposition \ref{InvalidValidTraj} tells us that the map $\mathcal{AA}$ can turn an invalid sequence of billiard collisions to a valid sequence of billiard collisions whose billiard trajectories are time- or light-like. The reverse is also true, though not of interest. What is of interest are the images of space-like trajectories under this map. As the $\mathcal{AA}$ map sends trajectories which reflect off of exactly one component curve to trajectories which alternate reflecting off of each component curve (or vice-versa), space-like trajectories inside the collared $\Hyp$-ellipse will be mapped one-to-one to space-like trajectories outside the collared $\Hyp$-ellipse. Of particular interest is when such trajectories are periodic.

\begin{theorem} \;
\begin{enumerate}[(i)]
    \item Suppose $\{(y_k,y_{k+1})\}$ is a space-like $2m$-periodic billiard trajectory inside the collared $\Hyp$-ellipse. Then the image of this sequence of collisions under the alternating antipodal map $\mathcal{AA}$ is either a $2m$- or  $m$-periodic trajectory outside the collared $\Hyp$-ellipse. 
    \item If $\{(y_k,y_{k+1})\}$ is a space-like $2m$-periodic billiard trajectory outside the collared $\Hyp$-ellipse, then the image of this sequence of collisions under the map $\mathcal{AA}$ is a $2m$-periodic orbit inside the collared $\Hyp$-ellipse.
    \item If $\{(y_k,y_{k+1})\}$ is a space-like $2m+1$-periodic billiard trajectory outside the collared $\Hyp$-ellipse, then the image of two concatenated copies of this sequence of collisions under the map $\mathcal{AA}$ is a $2(2m+1)$-periodic orbit inside the collared $\Hyp$-ellipse.
\end{enumerate}
\end{theorem}

\begin{figure}[htbp]
\begin{tabular}{c c}
a) \includegraphics[width=0.45\textwidth]{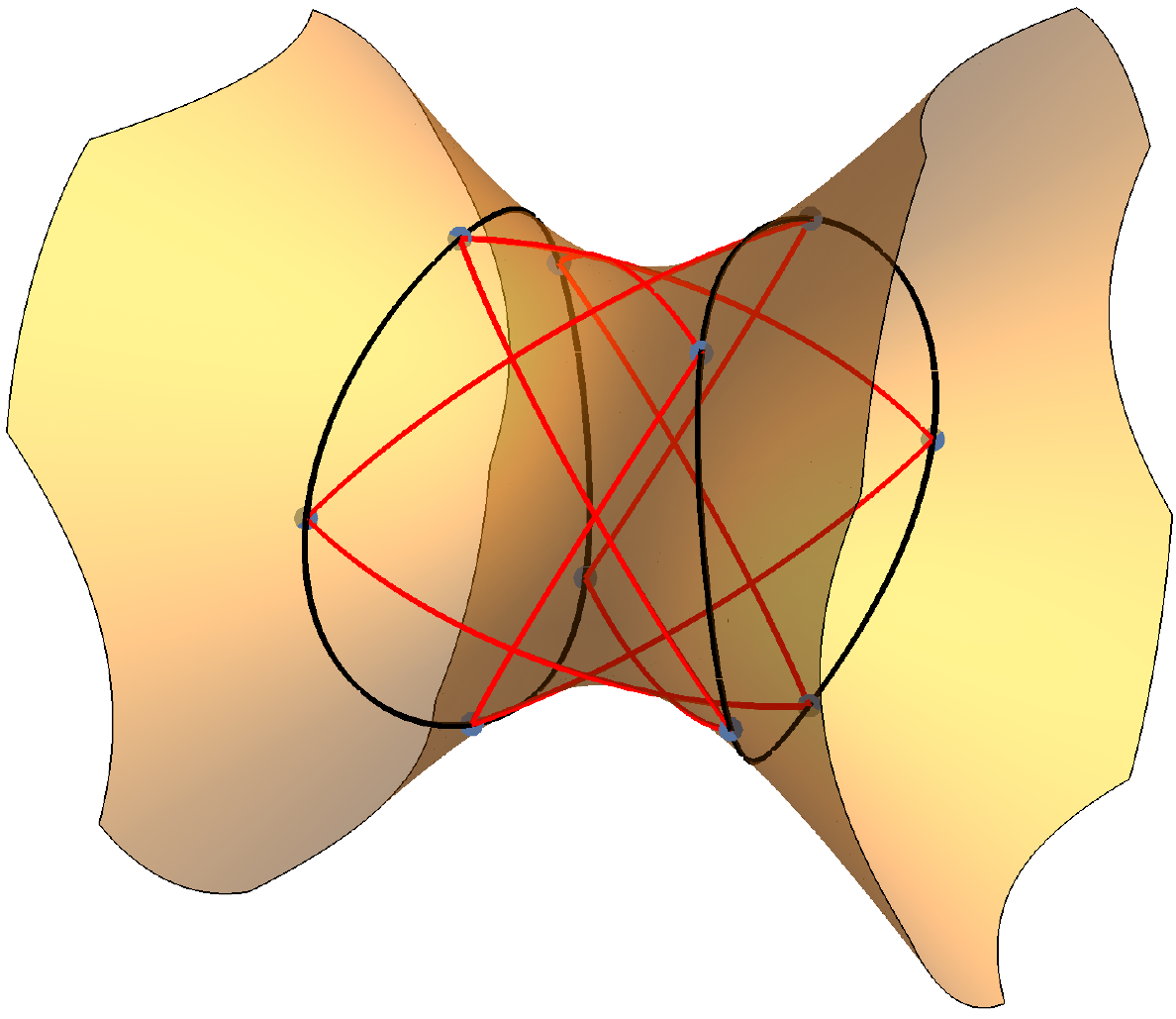} & b) \includegraphics[width=0.45\textwidth]{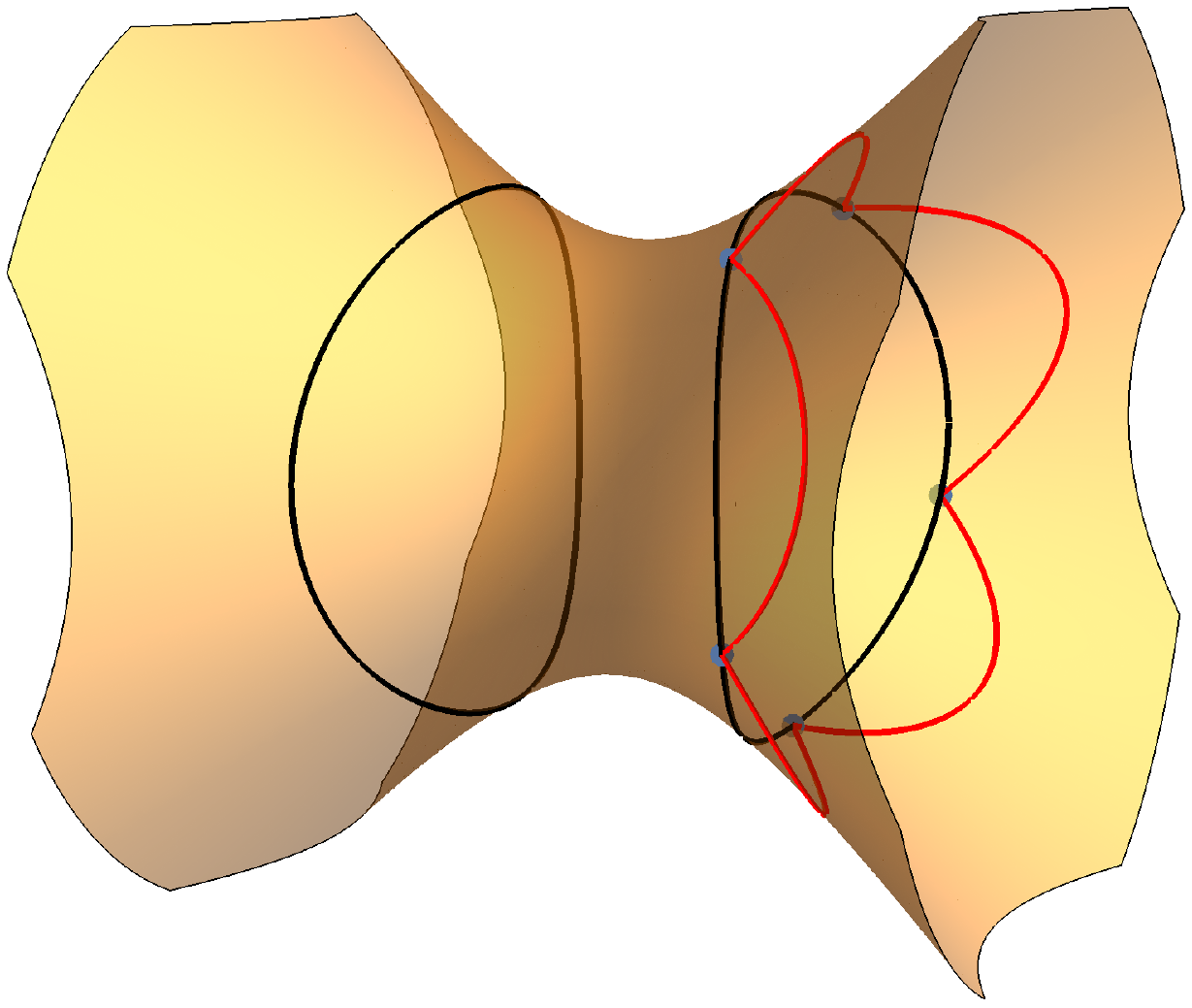} 
\end{tabular}
\caption{(a) A period 10 orbit inside the collared $\Hyp$-ellipse; (b) the image of this period 10 orbit under the $\mathcal{AA}$ map is a period 5 orbit outside the collared $\Hyp$-ellipse.}
\label{AAMapExamples}
\end{figure}

\section{Spectral Curves, Cayley's Condition, and Periodic Orbits}
\label{SpectralCurves}

Consider the spectral curve $\Gamma$ given by equation (\ref{spec1}), which can be rewritten in the following way
\begin{equation}
\label{spec2}
\Gamma:\;\det(L(\lambda) - \mu I)  = \det(A + \lambda x\wedge y - \mu I) =0.
\end{equation}
Using equations (\ref{spec1}) and (\ref{phieqn}) this can be reformulated as
\begin{equation}
\label{spec3}
\Gamma: p(\mu) - \lambda^2|x\wedge y|^2 q(\mu) =0
\end{equation} 
where %$$\displaystyle p(\mu) = \prod_{i=0}^2 (\mu - a_i),\;\;\; q(\mu) = \prod_{j=1}^{n-1}(\mu-\mu_j)$$
$$\displaystyle p(\mu) = (\mu - a_0)(\mu - a_1)(\mu - a_2),\;\;\; q(\mu) = \mu-\nu$$
and $\nu$ is the root of equation (\ref{phieqn}), $\phi_\mu(x,y) =0$. 

%\begin{remark}
%While the equation for $p(\mu)$ is the same as in \cite{V}, the equation of $q(\mu)$ is different in that $\deg(q) = n-1$. This is due to the fact that the leading term of $\phi_{\mu}(x,y)=0$ seen as a polynomial in $\mu$ has coefficient $\sum F_k$ which is 0 by the corollary above. The proof of this fact (stated as a proposition below) is similar to the analogous lemma in \cite{V} and Theorem 4.8 of \cite{KT}. 
%\end{remark}

\begin{prop}\label{PhiRoots}
Let $x,y \in \Hyp$. Then $\phi_\mu(x,y)=0$ has exactly one real root. This root can be written explicitly as 
\begin{align}
\nu &= \frac{a_1 a_2 F_0+a_0 a_2 F_1+a_0 a_1 F_2}{a_2 \left(F_0+F_1\right)+a_1 \left(F_0+F_2\right)+a_0 \left(F_1+F_2\right)} \nonumber\\
 &= \frac{-a_0 \left(x_1 y_2-x_2 y_1\right)^2+a_1 \left(x_0 y_2-x_2 y_0\right)^2+a_2 \left(x_0 y_1-x_1 y_0\right)^2}{-\left(x_1 y_2-x_2 y_1\right)^2+\left(x_0 y_2-x_2 y_0\right)^2+ \left(x_0 y_1-x_1 y_0\right)^2} 
 \label{LoneRootEqn}
\end{align}
where $x_i,y_j$ are the components of the points $x,y$. 
In particular, the straight line $xy$ on the hyperboloid of one sheet $\Hyp$ is tangent to the confocal conic  (\ref{ConfFam1}) corresponding to $\lambda = \nu$. This property and the equation for $\nu$ are preserved under the map $\mathcal{AA}$. 
\end{prop}

%Version of above proposition when considering this case in n dimensions
%\begin{prop}\label{PhiRoots}
%Let $x,y \in H$. Then $\phi_\mu(x,y)=0$ has $n-1$ or $n-3$ real roots in $\mu$. More specifically, if $x,y$ are points in the collared $\Hyp$-ellipse then $\phi_\mu(x,y)=0$ has exactly $n-1$ real roots; if $x,y$ are points in the transverse $\Hyp$-ellipse then $\phi_\mu(x,y)=0$ can have either $n-1$ or $n-3$ real roots. The straight line $xy$ on the hyperboloid of one sheet $\Hyp$ is tangent to the confocal conics (\ref{ConfFam}) corresponding to $\alpha = \mu_i$, $1 \leq i \leq n-3$ or $n-1$, as appropriate. 
%\end{prop}

Knowing the degree of $\phi_\mu$ now allows us to prove the theorem below, an analogue of a well-known theorem in Euclidean and pseudo-Euclidean geometry. 

\begin{theorem}
All segments of the billiard trajectory in the collared and transverse $\Hyp$-ellipse are tangent to the same confocal conic corresponding to $\lambda = \nu$.
This \emph{caustic} is fixed for a given trajectory and is invariant under the map $\mathcal{AA}$.
\label{FixedCaustic}
\end{theorem}

The proof of this is similar to that of Theorem 3 in \cite{V}, though with the appropriate adjustments due to Propositions \ref{NumRoots} and \ref{PhiRoots}. 

The work of \cite{V, MV, DJR} and others describe how to use the factorization procedure outlined in the previous section to compute eigenvectors of $\Gamma$ along with other spectral properties. We provide a brief summary below. 

Let $P_\beta$ be a point of $\Gamma$ with $\mu = \beta$ and let $P_{\pm}$ be the ``infinities" $\mu \approx \pm \lambda |x\wedge y|$ for $\lambda \to \infty$. 
 The eigenvector $\psi$ of $L(\lambda)$ normalized by the condition $\psi^0 + \psi^1 + \psi^2=1$ is the meromorphic vector function on $\Gamma$ with pole-divisor $\mathcal{D}$ of degree $3+g-1 = 3$, where $g= 1$ is the genus of $\Gamma$. By Theorem \ref{Thm1Eqn1} we can express the eigenvector $\psi_{k+1}$ of $L_{k+1}$ in terms of $\psi_k$ by 
\begin{equation}
\label{EigenvectorEqn}
\psi_{k+1} = A_k \psi_k = (A- \lambda(\zeta_k \otimes x_k^* + x_k \otimes \eta_k^*))\psi_k.
\end{equation}

Thus $\psi_{k+1}$ has two new poles $P_\pm$ and a new double zero at the point $Q_+$ corresponding to $\mu=0, \lambda = \ip{x}{A^{-1}y}^{-1} = \ip{x_{k-1}}{A^{-1}x_k}^{-1}$. This means the pole-divisor $\mathcal{D}_{k+1}$ of $\psi_{k+1}$ can be written as
\begin{equation}
\label{divisoreqn}
\mathcal{D}_{k+1} = \mathcal{D}_k + \mathcal{U} %P_+ + P_- -  2Q_+
\end{equation}
where $\mathcal{U}=P_+ + P_- - 2Q_+ = Q_- - Q_+$, where $Q_-$ has coordinates $\mu=0$, $\lambda = -\ip{x}{A^{-1}y}^{-1}$. The equivalence $P_+ + P_- = Q_+ - Q_-$ is given by the function $f(\mu, \lambda) = \mu$. This shift in the divisor $\mathcal{D}_k$ on $\Gamma$ corresponds to the points of reflection from the boundary in our billiard system. This is in fact Theorem 2 of \cite{V}, that the dynamics of the collared and transverse $\Hyp$-ellipse billiard problem correspond to the shift (\ref{divisoreqn}) on the Jacobi variety of the elliptic curve (\ref{spec2}).

Given a periodic billiard trajectory in the collared or transverse $\Hyp$-ellipse, it is known that trajectories with the same caustics have the same spectral curve. Thus the trajectory is of period $m$ if and only if $m(Q_- - Q_+) =0$  on the Jacobi variety $Jac(\Gamma)$. This proves the existence of a Poncelet-like result in this setting.

\begin{prop}
Given a periodic billiard trajectory in the collared or transverse $\Hyp$-ellipse, any billiard trajectory which shares the same caustic is also periodic with the same period. 
\end{prop} 

The work of Cayley (see \cite{C1,C2}, amongst many others) in the $19^{th}$ century and Griffiths and Harris \cite{GH} in the 1970's on the Poncelet Theorem lead to analytic conditions relating the period of a billiard trajectory to its caustics. Dragovi\'c and Radnovi\'c have proved such conditions for a Poncelet theorem the ellipsoid in $\R{d}$ \cite{DR1,DR2} and in Lobachevsky space \cite{DJR}. In particular, the techniques used apply here.

\begin{lemma}[\cite{DJR}]
Suppose the hyperelliptic curve $\Gamma$ is of the form $$\Gamma:\;\; y^2 = (x-x_1)\cdots(x-x_{2g+2})$$ with distinct nonzero $x_i$, $g$ is the genus of $\Gamma$, and $Q_+$ and $Q_-$ represent two points on $\Gamma$ over the point $x=0$. Then $m(Q_+-Q_-)=0$ is equivalent to 
$$\text{rank} \fourbyfour{B_{g+2}}{B_{g+3}}{B_{m+1}}{B_{g+3}}{B_{g+4}}{B_{m+2}}{B_{g+m}}{B_{g+m+1}}{B_{2m-1}}<m-g$$ for $m>g$ where $y=\sqrt{(x-x_1)\cdots(x-x_{2g+2})} = B_0 + B_1 x + B_2 x^2 + \cdots$ is the Taylor expansion around the point $Q_-$. 
\label{ImportantLemma}
\end{lemma}

Introduce the variable change $X =\mu$, $Y = \lambda|x\wedge y| q(X)$. This transforms equation (\ref{spec3}) into 
\begin{equation}
\label{spec4}
Y^2 = \ep(X-a_0)(X-a_1)(X-a_2)(X-\nu)
\end{equation}
where $\ep = \sgn{a_1\nu}$ and $\nu$ is the lone root of $\phi_\mu(x,y)$ described in proposition \ref{PhiRoots}. 

\begin{theorem}
The billiard trajectories in the collared and transverse $\Hyp$-ellipse with nondegenerate caustic $\mathcal{C}_\nu$ are $n$-periodic if and only if 
\begin{align}
m(Q_- - Q_+) &=0  \;\; (n=2m) \label{EvenDiv}\\
(m+1)Q_+ -mQ_- - P_\nu &=0 \;\; (n=2m+1) \label{OddDiv}
\end{align}
on the elliptic curve (\ref{spec4}), with $Q_\pm$ being two points on the curve over $X=0$ and $P_{\nu}$ is a point over $X=\nu$.
\label{DivisorThm}
\end{theorem}

\begin{proof}
Let $\mathcal{P}(x) = \ep(x-a_0)(x-a_1)(x-a_2)(x-\nu)$ where $\ep = \sgn{a_1\nu}$. Recall that every generic point on the hyperboloid of one sheet has two generalized Jacobi coordinates $\lambda_1, \lambda_2$ which satisfy inequalities outlined in proposition \ref{NumRoots}. 

First consider the case of the collared $\Hyp$-ellipse, which we denote by $\mathcal{E}$. The generalized Jacobi coordinates $\lambda_1, \lambda_2$ satisfy $\lambda_1 \leq a_0 < a_1 \leq \lambda_2 \leq a_2$. By Proposition \ref{ConfocalCurves}, the curve $\mathcal{C}_{\lambda_1}$ will be of elliptic-type for $\lambda_1 \in (-\infty,a_0]$ and $\mathcal{C}_{\lambda_2}$ will be hyperbolic-type for $\lambda_2 \in [a_1,a_2]$. %Along the trajectory the endpoints of the preceding intervals will be the only local extrema. 
Geometrically, $\lambda_1=0$ corresponds to the reflection of the trajectory off of the collared $\Hyp$-ellipse, $\mathcal{C}_0$; $\lambda_1 = a_0$ corresponds to the trajectory crossing the $x_0=0$ plane (which is also the plane of symmetry of the $\mathcal{E}$); $\lambda_2 = a_1$ and $a_2$ correspond to the trajectory crossing the $x_1=0$ and $x_2=0$ planes, respectively. 

There are three possibilities for types of trajectories:
\begin{enumerate}
    \item The caustic is of elliptic type outside of $\mathcal{E}$ and the billiard is within $\mathcal{E}$. Then $\nu < 0$ and $(\lambda_1,\lambda_2) \in [0,a_0]\times [a_1,a_2]$. There is no caustic inside $\mathcal{E}$, and $\lambda_1$ will take on the value 0 at each reflection point. Each coordinate plane must be crossed an even number of times. 
    \item The caustic is of elliptic type outside of $\mathcal{E}$ and the billiard is outside of $\mathcal{E}$. Then in accordance with the discussion in section 3, the billiard trajectories are space-like and all reflect off of one component of $\mathcal{E}$. Then $\nu<0$ and $(\lambda_1,\lambda_2) \in [\nu, 0]\times[a_1,a_2]$. The billiard moves between the one component of $\mathcal{E}$ and the caustic, will not cross the coordinate plane $x_0=0$, but must cross the coordinate planes $x_1=0$ and $x_2=0$ an even number of times. This is the only case which could have an odd period. 
    \item The caustic is of hyperbolic type and the billiard is inside $\mathcal{E}$. Then the caustic is symmetric about the plane $x_2=0$ and $\nu \in [a_1,a_2]$, so that $(\lambda_1,\lambda_2)\in [0,a_0]\times[a_1,\nu]$. The trajectory must become tangent to the caustic at some point inside $\mathcal{E}$.
\end{enumerate}
In each case above, the parameters $\lambda_1, \lambda_2$ change monotonically between the endpoints of the specified intervals. 

Following Jacobi \cite{Jacobi} consider the following differential equation along a billiard trajectory:
\begin{equation}
\frac{d\lambda_1}{\sqrt{\mathcal{P}(\lambda_1)}} + \frac{d\lambda_2}{\sqrt{\mathcal{P}(\lambda_2)}}=0.
\end{equation}
Let $P_\beta$ be a point over $X=\beta$ of the elliptic curve (\ref{spec4}) and $P_\pm$ are the points at infinity. The points $P_{a_0}$, $P_{a_1}$, $P_{a_2}$, and $P_{\nu}$ are all branching points and hence $2kP_{a_0} = 2kP_{a_1} = 2kP_{a_2} = 2kP_{\nu} = k(P_+ + P_-)$.
For a period $n$ trajectory, integrating the above equation along the trajectory leads to 
\begin{align*}
m_0 (P_{a_0}-P_0)  + m_1 (P_{a_2}-P_{a_1}) &=0 \\
m_2 (P_\nu-P_0) + m_3 (P_{a_2}-P_{a_1}) &=0 \\
m_4 (P_{a_0}-P_0) + m_5 (P_\nu - P_{a_1}) &=0 
\end{align*}
via the Abel map. Each $m_i$ must be even except for possibly $m_2$ and $n = m_0 = m_2 = m_4$. Using the equivalence $P_+ + P_- = Q_+ - Q_-$ and letting $P_0 = Q_+$, these three conditions reduce to 
\begin{align*}
\frac{n}{2}(Q_- - Q_+)&=0 \;\; (n \text{ even}) \\
\frac{n+1}{2}Q_+ - \frac{n-1}{2}Q_- - P_\nu &=0  \;\;\;(n \text{ odd})
\end{align*}
Writing $n=2m$ or $n=2m+1$ proves the theorem in the case of the collared $\Hyp$-ellipse.

Next, consider the case of the transverse $\Hyp$-ellipse which we denote by $\mathcal{T}$. The generalized Jacobi coordinates $\lambda_1, \lambda_2$ now satisfy $a_1 \leq \lambda_1 \leq \lambda_2 \leq a_0$.  Again by proposition \ref{ConfocalCurves}, if $\nu < a_1$ or $\nu > a_2$, then $\mathcal{C}_\nu$ will be of hyperbolic type and will not intersect $\mathcal{T}$; if $a_1 < \nu < a_0$ then $\mathcal{C}_\nu$ will be of elliptic type and intersect $\mathcal{T}$; if $a_0 < \nu < a_2$, then $\mathcal{C}_\nu$ will be of elliptic type but it will not intersect $\mathcal{T}$; and if $\nu \in \{a_0,a_1,a_2\}$ then $\mathcal{C}_\nu$ is degenerate and will be a hyperbola, a circle, or a hyperbola in the planes $x_0=0$, $x_1=0$, and $x_2=0$, respectively (though the last case will not affect the billiard in $\mathcal{T}$). Both of the possible degenerate cases will intersect $\mathcal{T}$ in two places each, and represent the two ``diameters" of $\mathcal{T}$.

There are three possibilities for trajectories in $\mathcal{T}$:
\begin{enumerate}
    \item The caustic is of hyperbolic type or elliptic type and does not intersect $\mathcal{T}$. Then $\nu < a_1$ or $\nu > a_2 $ (hyperbolic type) or $a_0 < \nu < a_2$ (elliptic type), and $(\lambda_1,\lambda_2) \in [a_1,0]\times [0,a_0]$. At each reflection point one coordinate takes on the value 0. They can both equal to 0 only at the four points where $\mathcal{T}$ has a light-like tangent. At these points the reflection is counted twice. On a closed trajectory the number of reflections is equal to the number of crossings of the planes $x_0=0$ and $x_1=0$, and there must be an even number of crossings of the coordinate planes.
    \item The caustic is of elliptic type and has a nonempty intersection with $\mathcal{T}$ and $a_1<\nu<0$. Then $(\lambda_1, \lambda_2) \in [\nu,0]\times[0,a_1]$. The caustic is oriented along the plane $x_1=0$ and the trajectory must cross the plane $x_0=0$ an even number of times. 
    \item The caustic is of elliptic type and has a nonempty intersection with $\mathcal{T}$ and $0<\nu<a_0$. Then $(\lambda_1, \lambda_2) \in [a_1,0]\times[0,\nu]$. The caustic is oriented along the plane $x_0=0$ and the trajectory must cross the plane $x_1=0$ an even number of times. 
\end{enumerate}
Repeating similar calculations as the case of the collared $\Hyp$-ellipse leads to the same two divisor conditions, (\ref{EvenDiv}) and (\ref{OddDiv}). 
\end{proof}

\begin{theorem}
Consider a billiard trajectory in the collared or transverse $\Hyp$-ellipse. The trajectory is $n$-periodic with period $n=2m \geq 4$ if and only if
\begin{equation}
\label{EvenCayley}
\det\fourbyfour{B_{3}}{B_{4}}{B_{m+1}}{B_{4}}{B_{5}}{B_{m+2}}{B_{m+1}}{B_{m+2}}{B_{2m-1}}=0.
\end{equation}
The trajectory is $n$-periodic with period $n=2m+1 \geq 3$ if and only if
\begin{equation}
\det \fourbyfour{D_{2}}{D_{3}}{D_{m}}{D_{3}}{D_{4}}{D_{m+1}}{D_{m}}{D_{m+1}}{D_{2m}}=0. 
\label{OddCayley}
\end{equation}
For each case,
$$\sqrt{\ep(X-a_0)(X-a_1)(X-a_2)(X-\nu)} = B_0 + B_1X + B_2 X^2 + \cdots$$
and 
$$\sqrt{\frac{\ep(X-a_0)(X-a_1)(X-a_2)}{X-\nu}} = D_0 + D_1X + D_2 X^2 + \cdots$$
are the Taylor expansions around $X=0$ with $\ep = \sgn{a_1\nu}$. Furthermore, the only 2-periodic trajectories are contained in the planes of symmetry. 
\label{CayleyThm}
\end{theorem}

\begin{proof}
Using theorem \ref{DivisorThm} we consider the even and odd cases separately. If $n=2m$ then lemma \ref{ImportantLemma} with $g=1$ applies directly, proving the condition (\ref{EvenCayley}). If $n = 2m+1$, the divisor condition (\ref{OddDiv}) is equivalent to the existence of a meromorphic function with a pole of order $m$ at $Q_-$, a simple pole at $P_\nu$, and a unique zero of order $m+1$ at $Q_+$. One basis of the space of such functions $\mathcal{L}(mQ_- + P_\nu)$ is 
\begin{align}
\{1, f_1, \ldots, f_m\}
\end{align} 
where 
\begin{align}
f_k &= \frac{y-D_0 - D_1 x + \cdots + D_k x^k}{x^k}
\end{align} and the coefficients $D_i$ are given in the statement of the theorem. The existence of such a function is equivalent to condition (\ref{OddCayley}).
\end{proof}

As is noted in \cite{DJR}, adjustments to the proofs of the previous two theorems can be made to allow the case when the elliptic curve (\ref{spec3}) has singularities (i.e. when the constants $a_i$ and root of $\phi_\mu(x,y)$ are not distinct). % (which cannot occur when $a_0, \ldots, a_n, \mu_1, \ldots, \mu_{n-1}$ are distinct and nonzero). However, such singularities occur in the following cases (and combinations thereof): $a_i = \mu_j$ for some $i, j$; $a_i = a_j$ for some $i \neq j$; \textcolor{red}{or $\mu_i = \mu_j$ for some $i \neq j$}. 

\begin{remark}
Theorem \ref{CayleyThm} provides an analytical condition for  trajectories in the collared and transverse $\Hyp$-ellipse. In the case of the collared $\Hyp$-ellipse, there is a caveat to the interpretation of these periodic trajectories. As the collared $\Hyp$-ellipse has two components curves, any periodic trajectory inside the collared $\Hyp$-ellipse must have even period, but a periodic trajectory outside the collared $\Hyp$-ellipse can have odd period (see theorem \ref{DifferentBoundary} and corollary \ref{DifferentBoundary2}). That is, the condition given by (\ref{OddCayley}) will produce a trajectory with period $2(2m+1)$ or $2m+1$ if the motion is on the interior or exterior of the collared $\Hyp$-ellipse, respectively. Both cases will project to a trajectory with period $2m+1$ in the Klein coordinates $(\xi_1, \xi_2)$. 

We also note equation (\ref{EigenvectorEqn}) is invariant with respect to symmetries in the coordinate planes of the vectors $x_k$, $\eta_k$, and $\zeta_k$ which represent motion.  As such, closed trajectories will be unique up to those symmetries. 
\end{remark}

In \cite{DR3,DR4}, periodic light-like trajectories inside ellipses in the Minkowski plane are studied. We may also consider the previous two theorems in the special case of light-like trajectories. On $\Hyp$, a general light-like trajectory is a member of one of two families of generatrices. Upon reflection from the boundary of the collared or transverse $\Hyp$-ellipse, the billiard trajectory will switch from one family to the other. Thus to be periodic, light-like trajectories must be of even period, and we only need to consider condition (\ref{EvenCayley}) in this context. As light-like trajectories have a caustic at $\nu=\infty$, we arrive at a similar theorem.

\begin{theorem}
Consider a light-like billiard trajectory in the collared or transverse $\Hyp$-ellipse. The trajectory is $2m$-periodic for $2m \geq 4$ if and only if
\begin{equation}
\label{LightEvenCayley}
\det\fourbyfour{E_{3}}{E_{4}}{E_{m+1}}{E_{4}}{E_{5}}{E_{m+2}}{E_{m+1}}{E_{m+2}}{E_{2m-1}}=0,
\end{equation}
where
$$\sqrt{\delta(X-a_0)(X-a_1)(X-a_2)} = E_0 + E_1X + E_2 X^2 + \cdots$$
is the Taylor expansion around $X=0$ and $\delta = \sgn{a_1}$. 
\label{LightCayleyThm}
\end{theorem}

%\begin{remark}
%Theorem \ref{CayleyThm} provides an analytical condition for even-period trajectories in the collared and transverse $\Hyp$-ellipse. Equation (\ref{EigenvectorEqn}) is invariant with respect to symmetries in the coordinate planes of the vectors $x_k$, $\eta_k$, and $\zeta_k$ which represent motion.  As such, closed trajectories will be closed up to those symmetries. Odd-periodic trajectories can exist in certain cases with appropriate adjustments. For example, an odd-periodic trajectory can exist in the setting of the collared $\Hyp$-ellipse if we allow the motion to be outside the collared $\Hyp$-ellipse, so the reflection only occurs on one component curve (and all such odd-period trajectories must be space-like). In the case of the transverse $\Hyp$-ellipse, one can derive a similar Cayley-type condition to Theorem 3.1 of \cite{DR4}. However the odd-period trajectories will not have the same symmetries as those derived by equation (\ref{CayleyCond}). 
%\end{remark}

\section{Geometric Consequences and Examples}\label{GC}

\subsection{The Collared $\Hyp$-Ellipse}
Consider the hyperboloid of one sheet in $\Mink$ with $A=\diag(3,6,9)$. The two component curves of the collared $\Hyp$-ellipse can be parametrized as $$\gamma_C(t) = \left(\pm \frac{\sqrt{2+t^2}}{2}, t,\pm \frac{\sqrt{3}}{2} \sqrt{2-t^2}\right), \;\; -\sqrt{2} \leq t \leq \sqrt{2}.$$ This is projected into the Klein coordinates as 
\begin{equation}
\label{A1Klein}
\frac{\xi_1^2}{2} + \frac{\xi_2^2}{3}=1.
\end{equation}
The projection $\pi_\xi$ is a double cover of the ellipse in the $\xi_1\xi_2$-plane. Any period $m$ trajectory in the $\xi_1\xi_2$-plane can correspond to either a period $m$ or period $2m$ trajectory on the hyperboloid of one sheet.

\begin{figure}[bht]
\begin{tabular}{c c}
a) \includegraphics[width=0.50\textwidth]{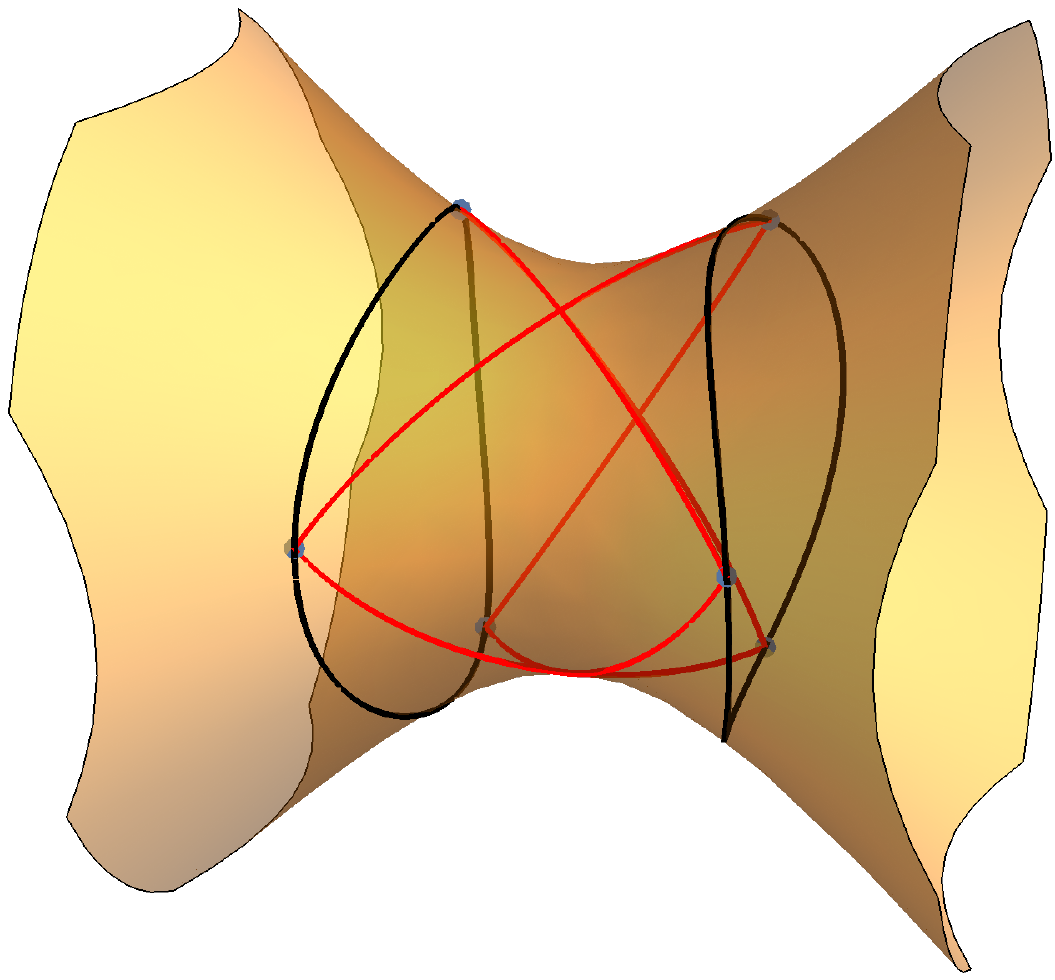} & b) \includegraphics[width=0.40\textwidth]{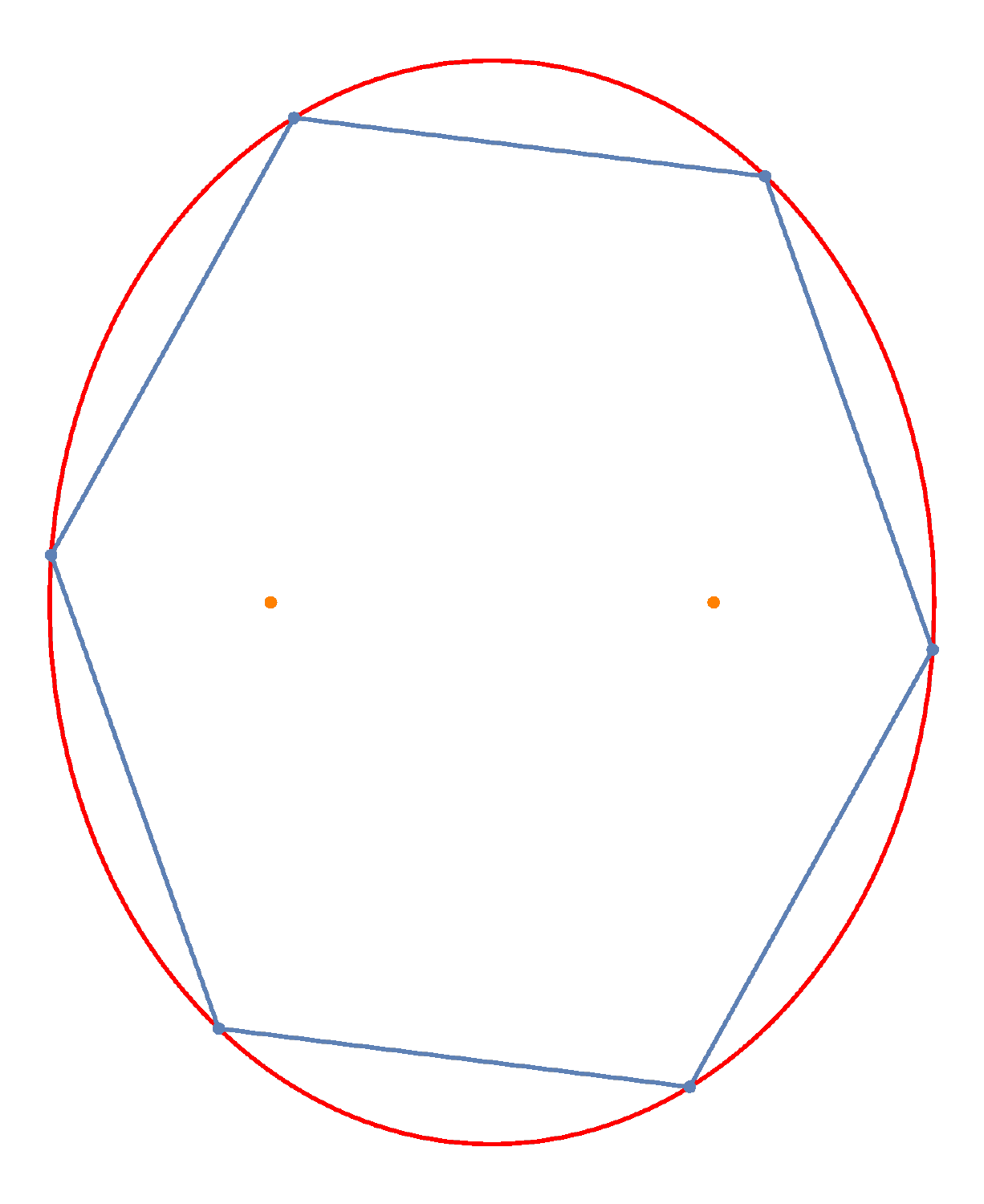} \\
\end{tabular}
\caption{a) A space-like period 6 trajectory in the collared $\Hyp$-ellipse. b) the projection of the period 6 orbit into Klein coordinates. }
\label{EllipseExample}
\end{figure}

An interesting geometric consequence due to the projection into Klein coordinates is that the ellipse (\ref{A1Klein}) should have foci at $(0,\pm 1)$ in the standard Euclidean sense, however the projected trajectories produce caustics which are confocal around the points $(\pm 1/\sqrt{2},0)$. See figure \ref{TrajEx}b. 

The caustic of a periodic trajectory can be of elliptic- or hyperbolic-type. The curve $\mathcal{C}_\nu$ projects into Klein coordinates as 
\begin{equation}
\frac{\xi_1^2}{\frac{a_1-\nu}{a_0-\nu}}+\frac{\xi_2^2}{\frac{a_2-\nu}{a_0-\nu}} =1%, \qquad \frac{\xi_1^2}{\frac{\nu-a_1}{\nu-a_0}}-\frac{\xi_2^2}{\frac{a_2-\nu}{\nu-a_0}} =1,
\label{ProjectedCausticEqn}
\end{equation}
which are Euclidean ellipses and hyperbolas (oriented along the $\xi_1$-axis) for $\nu \in (-\infty,a_0)$ and $\nu \in (a_1,a_2)$, respectively. In addition, for $\nu >a_2$ the hyperboloid of one sheet and the cone $\ip{(A-\nu I)^{-1}x}{x}=0$ do not intersect (i.e. $\mathcal{C}_\nu = \emptyset$), but the curve above will still be an elliptical caustic for the projected billiard trajectory.

\begin{cor}
Consider the billiard in the collared $\Hyp$-ellipse. 
\begin{enumerate}[(i)]
    \item Once projected into Klein coordinates $(\xi_1,\xi_2)$, the billiard table is itself an ellipse and the projection of the caustic curves $\mathcal{C}_\nu$ preserves the type of curve. 
    \item A trajectory inside the collared $\Hyp$-ellipse stays in the region $\{x_2 \geq 0\}$ (after suitable change of coordinates, if necessary) if and only if the trajectory projected into the Klein coordinates has a hyperbolic caustic. 
    \item Any trajectory in the Klein ellipse that passes through one focus will upon reflection pass through the other focus. 
\end{enumerate}
\end{cor}

Property (3) of billiards is well-known \cite{Tab} and was shown to also be true in \cite{V} for the cases of one sheet of the hyperboloid of two sheets and spherical billiards once projected into Klein coordinates. 

Another related corollary is an extension of the ``string construction" of the Euclidean ellipse, given by Graves \cite{Ber}.

\begin{theorem}[Graves]
Given an ellipse $E$ and a closed piece of string with length strictly greater than the perimeter of $E$, the locus of a pencil used to pull the string taut around $E$ is another ellipse, $E^\prime$, confocal with $E$. 
\end{theorem}

\begin{cor}
The Graves' Theorem also applies to the collared $\Hyp$-ellipse once projected into the Klein coordinates $(\xi_1,\xi_2)$. Moreover, the ellipse can be determined geometrically as the locus of all points $\xi$ satisfying $$\widehat{\rho}(F_-,\xi) + \widehat{\rho}(F_+,\xi) = c$$ for fixed foci $F_{\pm}$, constant $c$, and $\widehat{\rho}$ is the distance on the hyperboloid of one sheet in Klein coordinates.
\end{cor}

It is a quick exercise in differential geometry to see that the square of the arc length differential in $\Mink$ can be written in Klein coordinates as 
$$ds^2 = -dx^2 + dy^2 + dz^2 = \frac{\left(\xi _2^2-1\right)d\xi_1^2  - (2 \xi_1 \xi_2)d\xi_1 d\xi_2 +\left(\xi_1^2 - 1 \right)d\xi_2^2}{\left(\xi_1^2+\xi_2^2 - 1 \right)^2}.$$ 
This expression is just the negative of the square of the arc length differential of the Lobachevskian metric (i.e. $ds^2 = -ds_{Lob}^2$). 

\subsection{The Transverse $\Hyp$-Ellipse} Consider the hyperboloid of one sheet with $A = \diag(3,-3,6)$. The transverse $\Hyp$-ellipse can be parametrized as $$\gamma_T(t) = \left(t,\pm \frac{\sqrt{1-t^2}}{\sqrt{3}}, \sqrt{\frac{2}{3}}\sqrt{1+2t^2}\right), \;\; -1 \leq t \leq 1.$$ This is projected into the Klein coordinates as
\begin{equation}
\label{A2Klein}
\frac{\xi_2^2}{2} - \xi_1^2=1,
\end{equation}
though it should be noted that each of the two curves in $\gamma_T(t)$ project to two disjoint halves of the branches of the hyperbola (\ref{A2Klein}). The equation of the projected caustics  (\ref{ProjectedCausticEqn}) again preserves the type of the caustic curve $\mathcal{C}_\nu$.

\begin{figure}[htbp]
\begin{tabular}{c c}
a) \includegraphics[width=0.45\textwidth]{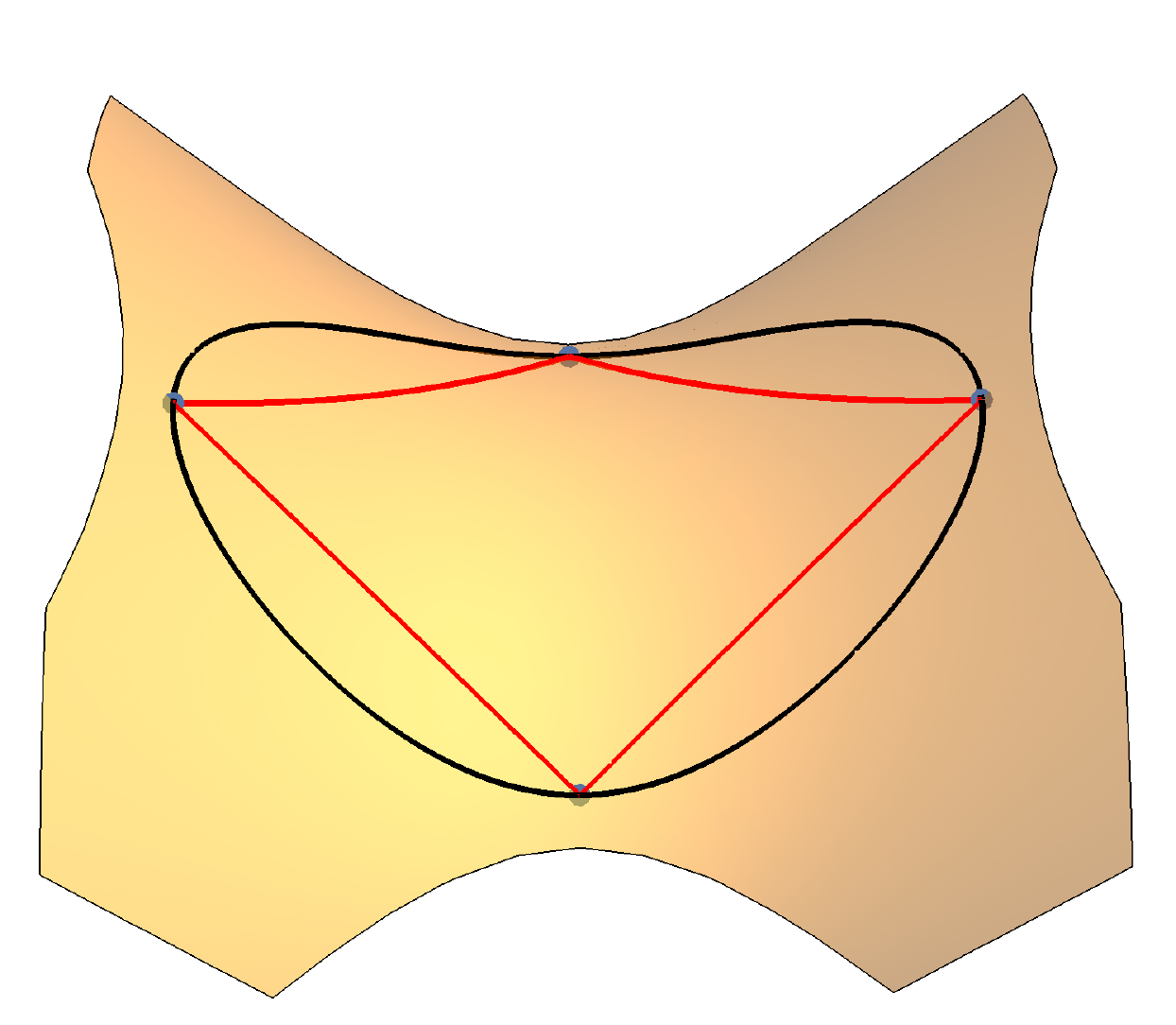} & b) \includegraphics[width=0.40\textwidth]{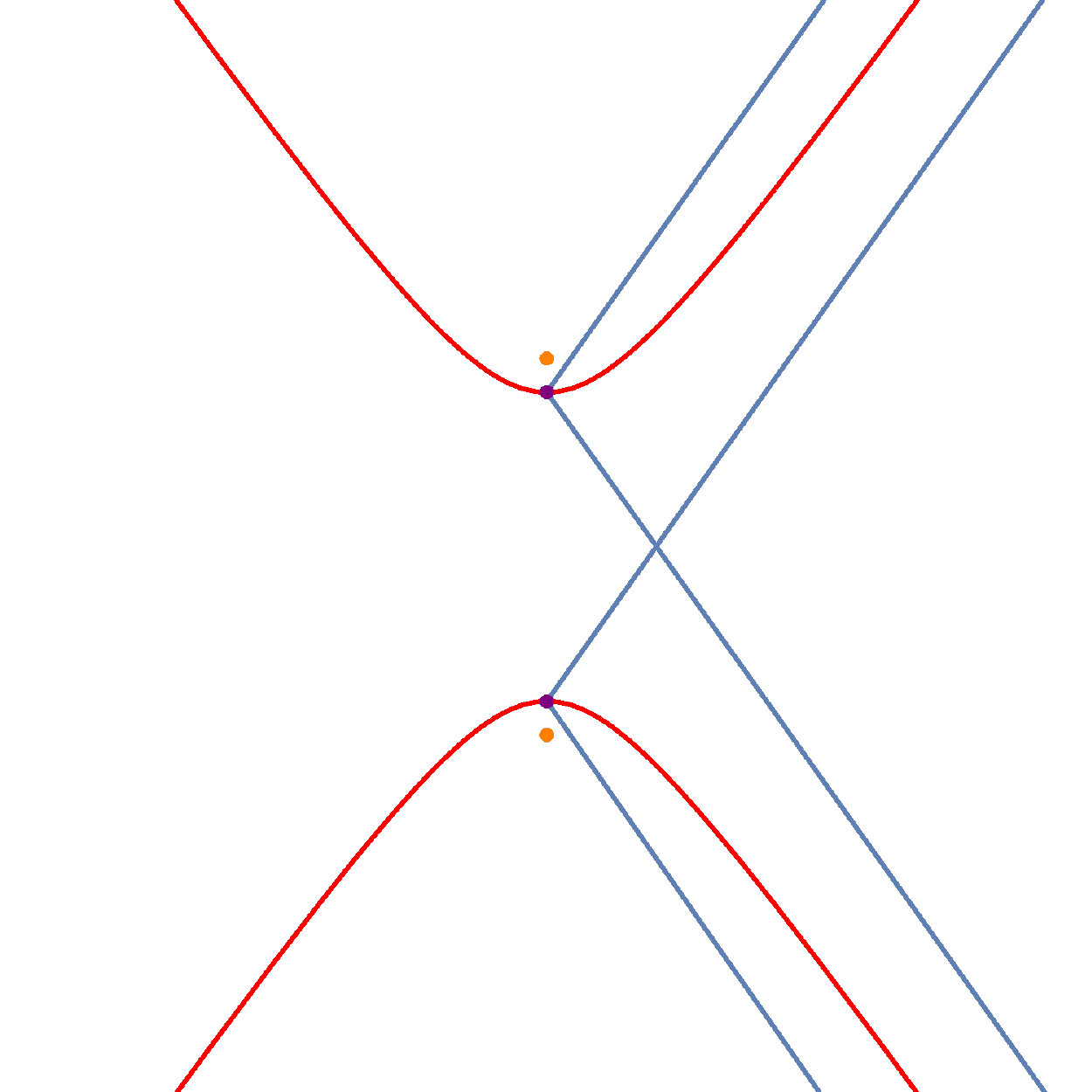} \\
\end{tabular}
\caption{a) A time-like period 4 trajectory in the transverse $\Hyp$-ellipse. b) the projection of the period 4 orbit into Klein coordinates to the point at infinity (shown as two separate locations to illustrate the periodicity of the orbit). } 
\label{HyperboloidExample1}
\end{figure}

Billiards in this Klein model look far different than the previous example. The billiard trajectory may  reflect from one branch of the hyperbola to the other or it can travel only along a single branch. A further disadvantage of this model is that points whose $x_0$ coordinate have opposite signs and are near the $x_0=0$ plane are projected to points that are far away from one another in the $\xi_1\xi_2$-plane  (i.e. $|\xi_1|, |\xi_2|\to\infty$) which make periodic orbits hard to visually identify.  For example, there is a 4-periodic orbit in the transverse $\Hyp$-ellipse which consists of the following points $$\left(-1,0,\sqrt{2}\right),\;\; \left(0,\frac{1}{\sqrt{3}},\sqrt{\frac{2}{3}}\right),\;\; \left(1,0,\sqrt{2}\right),\;\; \left(0,-\frac{1}{\sqrt{3}},\sqrt{\frac{2}{3}}\right).$$ %with caustic parameter $\nu = -6$. 
However the second and fourth points are projected to the point at infinity in the Klein model, see figure \ref{HyperboloidExample1}. 

\subsection{Periodic Trajectories from the Cayley Condition}

Let the matrix $A = \diag(a,b,c)$. The conditions (\ref{EvenCayley}) and (\ref{OddCayley}) can be used to classify periodic trajectories in terms of the parameters $a, b, c$ and the caustic parameter $\nu$.

\begin{example}[3-periodic trajectories]
The condition (\ref{OddCayley}) for a period 3 trajectory is for $D_2=0$. This is equivalent to 
$$3(abc)^2\ -2 abc\left(ab + bc + ac \right)\nu + \left( 4 a b c (a+b+c)-(a b+a c+b c)^2\right) \nu^2 = 0$$
which has solutions 
\begin{align*}
\nu_1,\nu_2 &= \frac{abc(a+b+c) \pm 2abc \sqrt{a^2 b^2 + a^2 c^2 + b^2 c^2 - a b c (a + b + c)}}{4 a b c (a + b + c) - (a b + b c + a c)^2}. %\\
%\nu_2 &= \frac{abc(a+b+c) + 2abc \sqrt{a^2 b^2 + a^2 c^2 + b^2 c^2 - a b c (a + b + c)}}{4 a b c (a + b + c) - (a b + b c + a c)^2}.
\end{align*}
This condition works for the transverse $\Hyp$-ellipse but if the two component curves of the collared $\Hyp$-ellipse are sufficiently far apart there cannot be a period 3 trajectory outside the collared $\Hyp$-ellipse due to the corresponding points not being geodesically connectable. In such a case, the condition above will produce period 6 trajectories inside the collared $\Hyp$-ellipse by using the $\mathcal{AA}$ map. 

\begin{figure}[htbp]
\begin{tabular}{c c}
a) \includegraphics[width=0.45\textwidth]{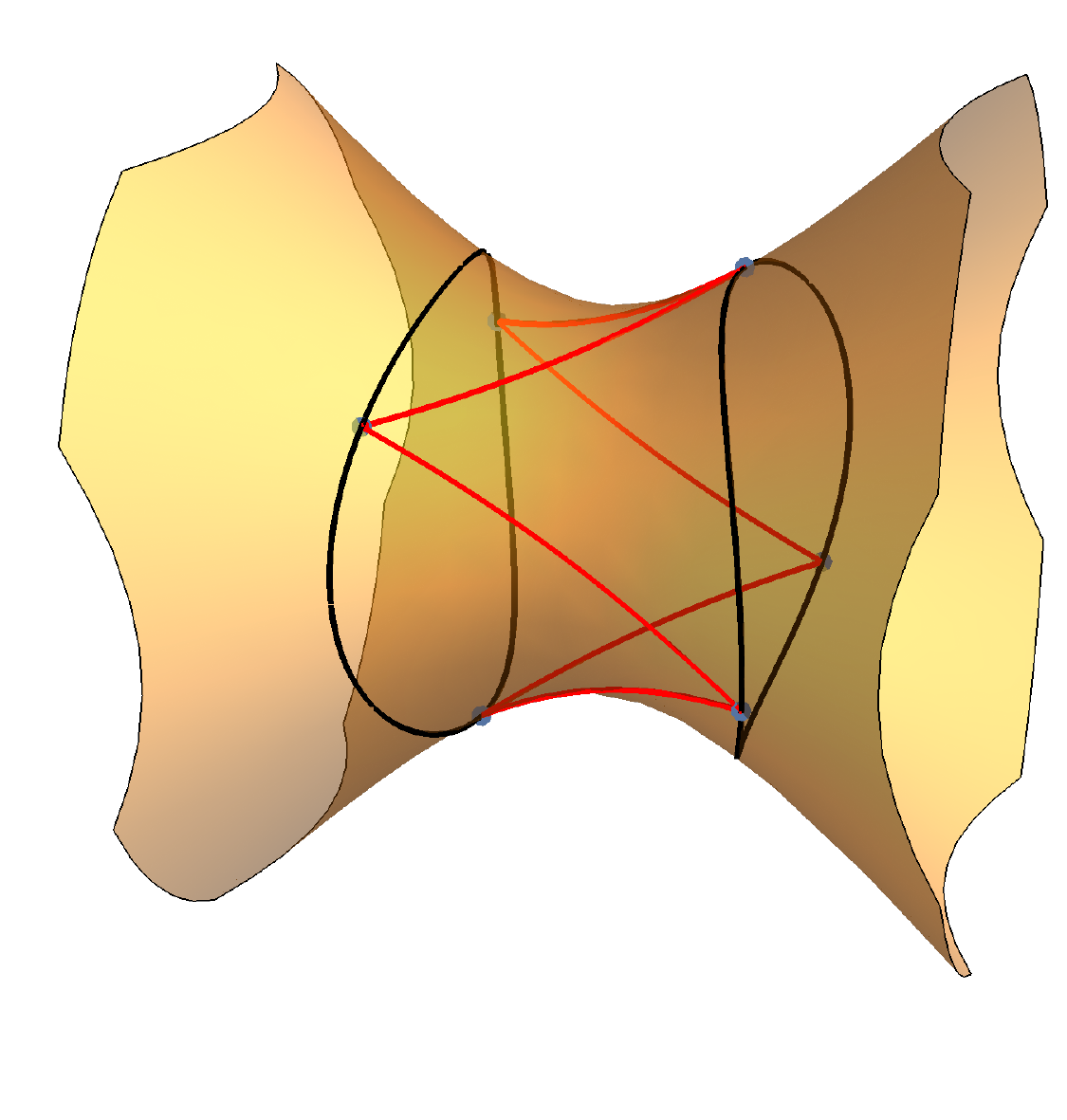} & b) \includegraphics[width=0.45\textwidth]{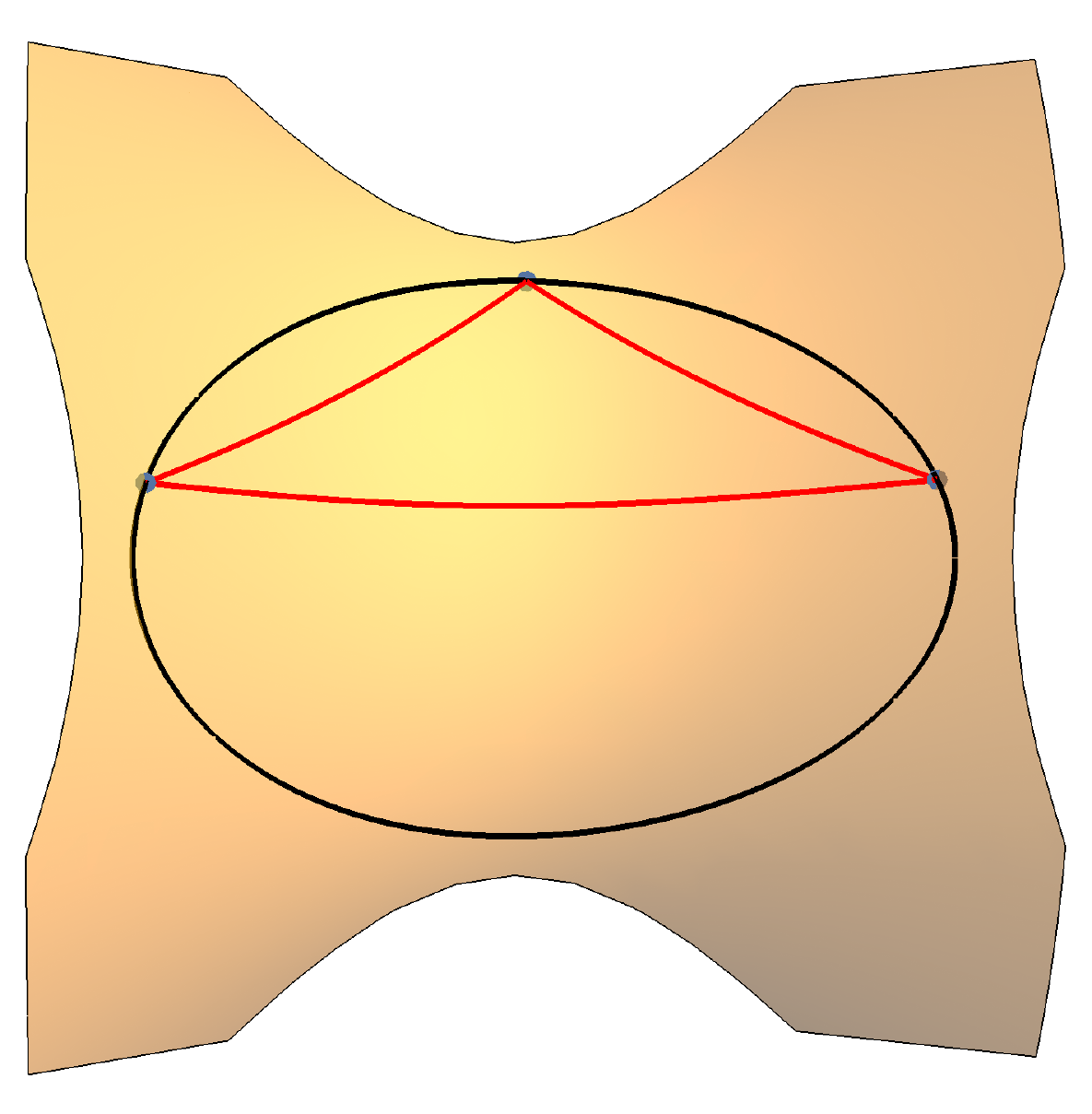}
\end{tabular}
\caption{Time-like period 6 and 3 trajectories in the collared and transverse $\Hyp$-ellipse, respectively, using the condition $D_2=0$.}
\end{figure}
\end{example}

\begin{example}[4-Periodic trajectories]\label{4PeriodicExample}
The condition (\ref{EvenCayley}) for 4-period trajectories is $B_3=0$, which is equivalent to 
$$(\nu  (-a b+a c+b c)-a b c) (\nu  (a b+a c-b c)-a b c) (\nu  (a b-a c+b c)-a b c) = 0.$$
The numerator is cubic in $\nu$ and has roots
$$\nu_1 =\frac{a b c}{-a b+b c+a c}, \;\; \nu_2 = \frac{a b c}{a b-b c+a c},\;\;\; \nu_3 = \frac{a b c}{a b+b c-a c}.$$
Using definition \ref{ElHyp}, we can state specifically when these roots are defined. In the case of the collared $\Hyp$-ellipse, the denominators of $\nu_1,\nu_3$ will never vanish. The denominator of $\nu_2$ will vanish if $(a,b,c) = (a, b, ab/(b-a))$ for $a<b<2a$. %The denominator of $\nu_2$ will vanish if the vector of constants $(a,b,c) = a(1, d, d/(d-1))$ for $1<d<2$. %this is the old way it was written.
In the case of the transverse $\Hyp$-ellipse, the denominator of $\nu_1$ will vanish if $(a,b,c) = (a,ac/(a-c),c)$. The denominators of $\nu_2,\nu_3$ will never vanish. 
%In the case of the transverse $\Hyp$-ellipse, the denominator of $\nu_1$ will vanish if the vector of constants $(a,b,c) = a(1,e,e/(e+1))$ for $e<-1$. The denominators of $\nu_2,\nu_3$ will never vanish. %Old version

\begin{figure}[h]
\begin{tabular}{c c} 
a) \includegraphics[width=0.45\textwidth]{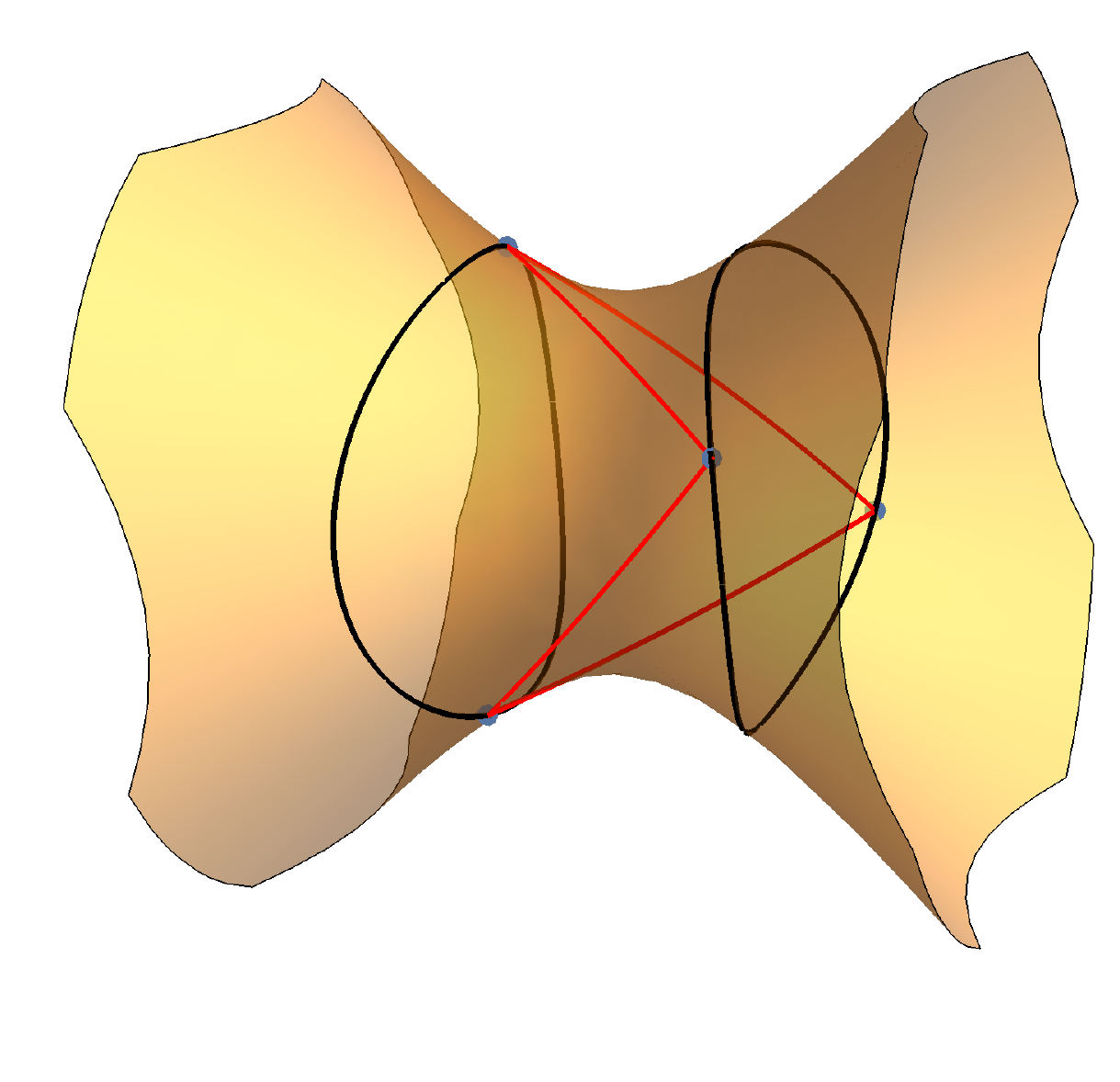} & b) \includegraphics[width=0.45\textwidth]{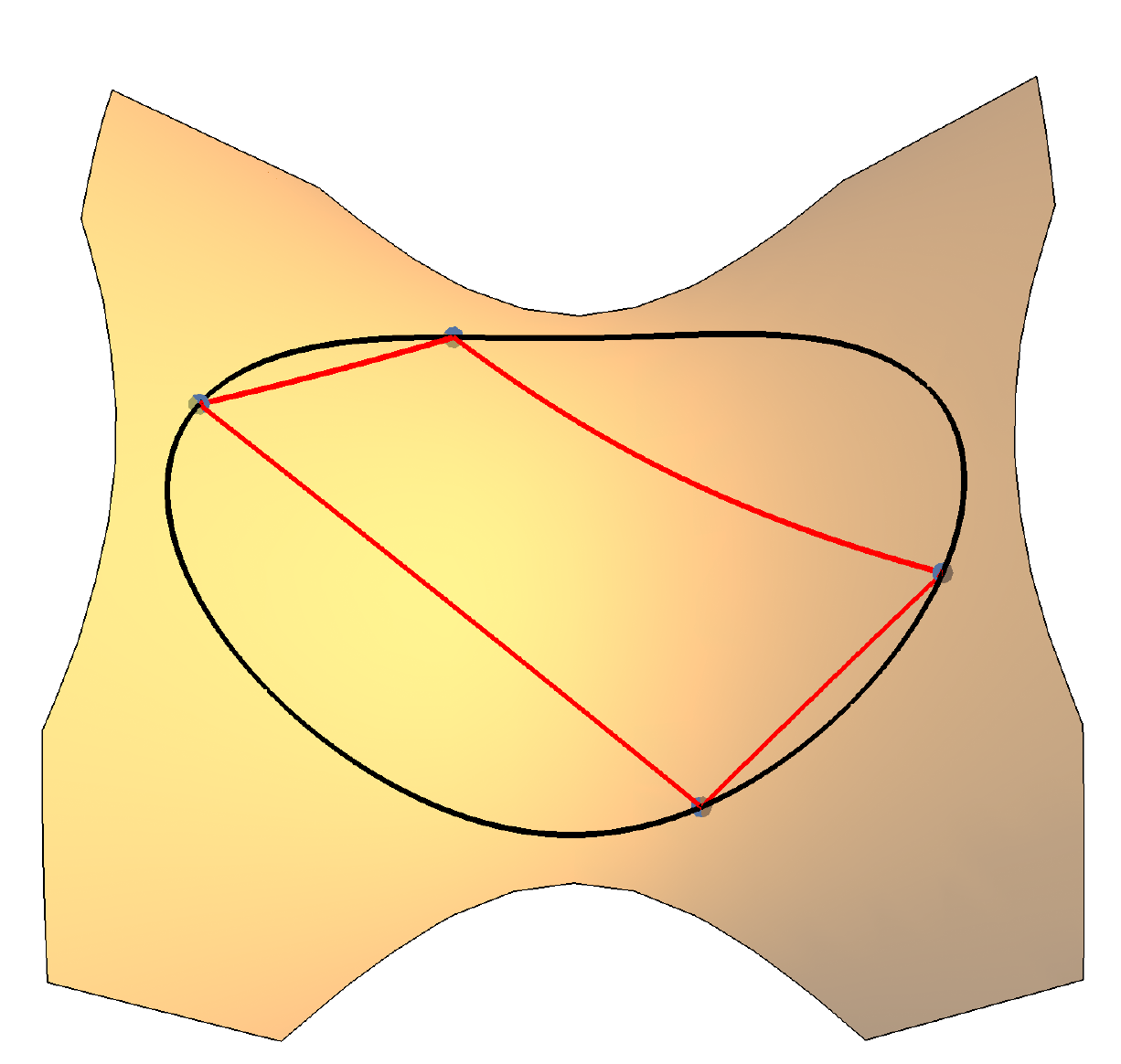}
\end{tabular}
\caption{Period 4 trajectories that are (a) light-like in the collared $\Hyp$-ellipse, and (b) time-like in the transverse $\Hyp$-ellipse, using the condition $E_3=0$ and $B_3=0$, respectively.}
\end{figure}

In both of the above cases, vanishing denominators correspond to 4-periodic light-like trajectories which are tangent to the caustic at infinity, $\nu = \infty$. This is consistent with the condition $E_3=0$ from (\ref{LightEvenCayley}).  %In fact, if $a,b,c$ satisfy $$(-a b + a c + b c) (a b - b c + a c ) (a b+ b c - a c )=0$$ then the trajectory will be light-like and period 4.
\end{example}

\begin{example}[5-periodic trajectories]
The condition $D_2 D_4 - D_3^2=0$ is equivalent to finding the roots of a degree 6 polynomial in $\nu$. Its simplest expression is given in terms of the elementary symmetric polynomials in 3 variables, $p := a b c, q := ab + ac + bc, r := a + b + c$:  
\begin{align*}
0 &= 5 r^6  -10 q r^5\nu + r^4 \left(52 p r-9 q^2\right)\nu^2 + 4 r^3 \left(-36 p q r+9 q^3+56 r^2\right) \nu^3 \\
 & \qquad + r^2 \left(-16 r^2 \left(p^2+14 q\right)+120 p q^2 r-29 q^4\right)\nu^4  \\
 & \qquad + 2 r \left(16 q r^2 \left(q-p^2\right)-8 p q^3 r+64 p r^3+3 q^5\right)\nu^5 \\
 & \qquad  + (48 p^2 q^2 r^2-64 r^3 \left(p^3+4 r\right)-12 p q^4 r+128 p q r^3-32 q^3 r^2+q^6)\nu^6
\end{align*}
In the collared $\Hyp$-ellipse with $a=3, b=6, c=9$, this produces four real roots and two imaginary roots: $$\nu \approx -4.39698, \; 2.06224, \; 2.99982, \; 9.39196.$$
In the transverse $\Hyp$-ellipse with $a=3, b=-3, c=6$, this produces four real roots and two imaginary roots: $$\nu \approx -2.99945, \; -1.26894, \; 0.741316, \; 2.87981.$$

\begin{figure}[h]
\begin{tabular}{c c}
a) \includegraphics[width=0.45\textwidth]{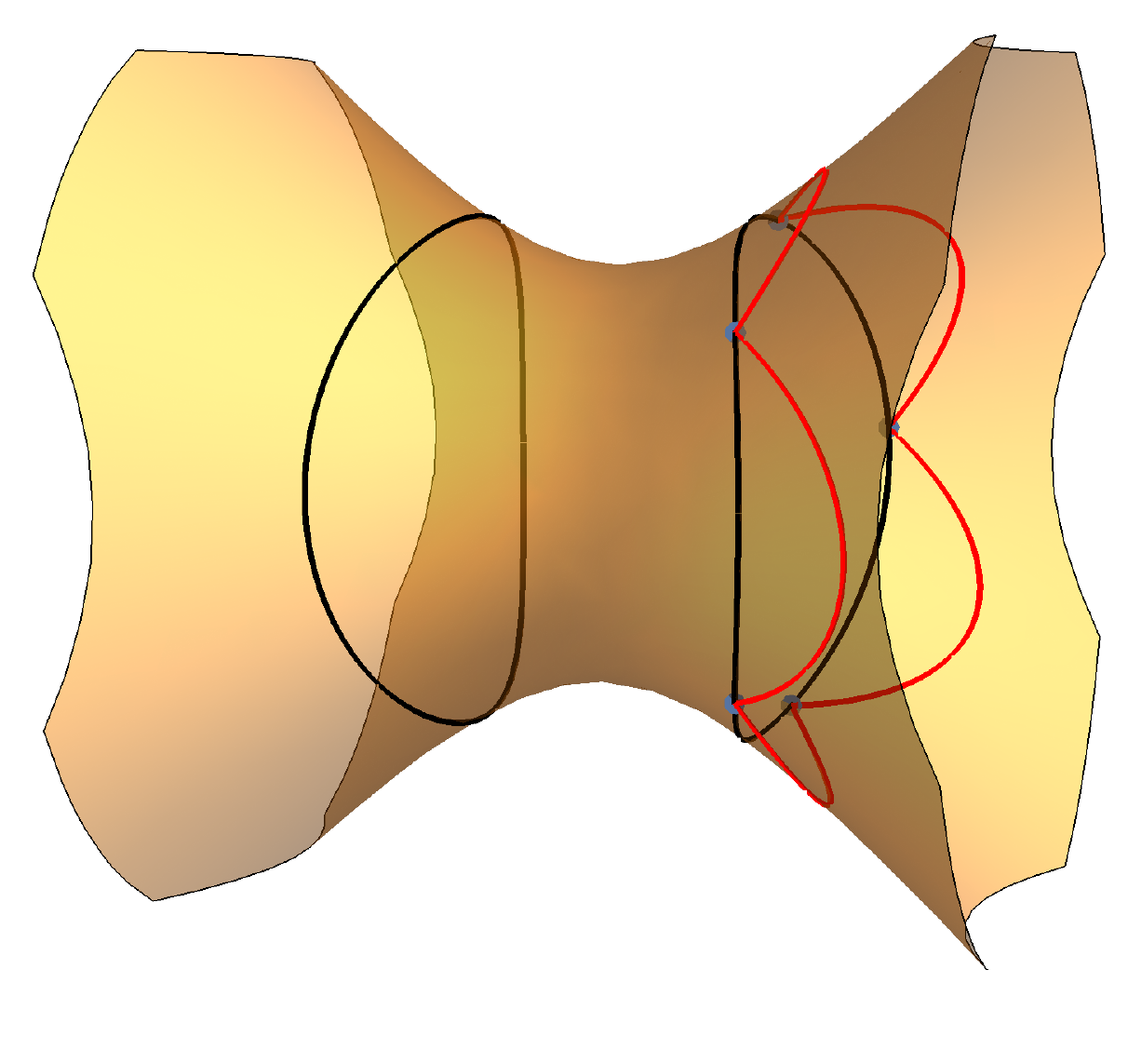} & b) \includegraphics[width=0.45\textwidth]{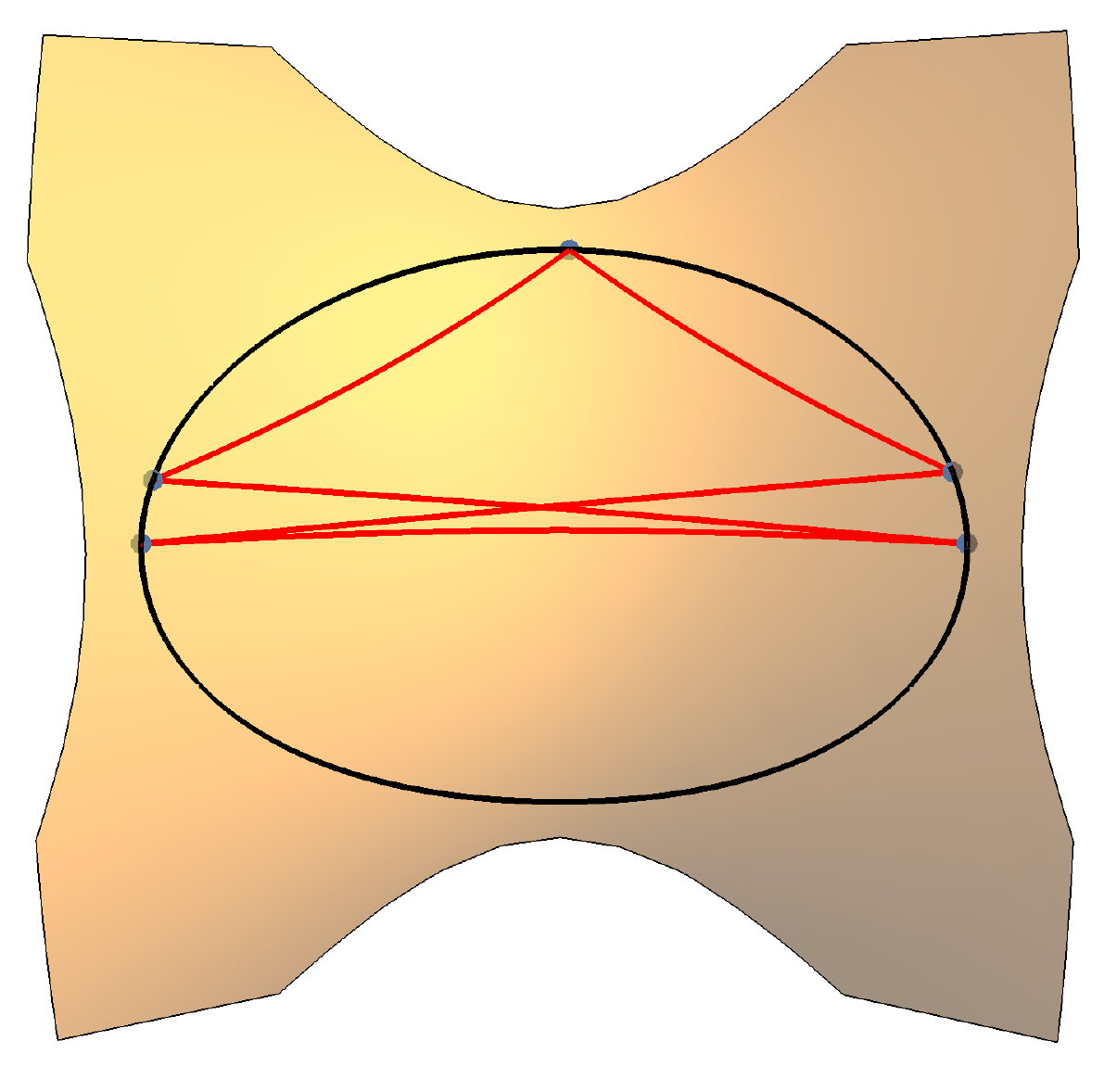}
\end{tabular}
\caption{Period 5 trajectories that are (a) space-like and outside the collared $\Hyp$-ellipse, and time-like and inside the transverse $\Hyp$-ellipse, using the condition $D_2D_4 - D_3^2=0$.}
\end{figure}
\end{example}

\begin{example}[6-periodic trajectories]
The condition for a time- or space-like period 6 orbit from (\ref{EvenCayley}) is that $B_3B_5-B_4^2=0$. This is equivalent to 
\begin{align*}
0&=\left[\left(-3 a^2 b^2+c^2 (a-b)^2+2 a b c (a+b)\right)\nu ^2 +2 a b c (a b-ac-bc) \nu+ (abc)^2\right] \\
&\times \left[(-a^2 (b-c)^2+2 a b c (b+c)-b^2 c^2)\nu^2 - 2abc(ab + ac + bc)\nu + 3(abc)^2\right] \\
&\times \left[(a^2 (b-c)^2+2 a b c (b+c)-3 b^2 c^2)\nu^2 + 2abc (-ab -ac +bc )\nu + (abc)^2\right] \\
&\times \left[(a^2 (b-c) (b+3 c)+2 a b c (c-b)+b^2 c^2)\nu^2 + 2abc(-ab + ac - bc)\nu + (abc)^2 \right]
\end{align*}

The first quadratic has discriminant $16a^3 b^3 c^2 (c-a) (c-b)$ which is positive for the collared $\Hyp$-ellipse and negative for the transverse $\Hyp$-ellipse. The roots are given by 
\begin{align*}
\nu_{1,2} &= \frac{a b c}{-ab + ac + bc \pm 2 \sqrt{a b (c-a) (c-b)}}.
\end{align*}
The second quadratic in the product above is equivalent to $D_2=0$, so it produces period 3 trajectories.
The third quadratic has discriminant $16 a^2 b^3 c^3 (b-a) (c-a)$ which is positive for both the collared and transverse $\Hyp$-ellipse. The roots are given by
\begin{align*}
\nu_{1,2} &= \frac{a b c}{ab+ac-b c \pm 2 \sqrt{b c (a-b) (a-c)}}.
\end{align*}
The fourth quadratic has discriminant $16 a^3 b^2 c^3 (a-b) (c-b)$ which is negative for the collared $\Hyp$-ellipse and positive for the transverse $\Hyp$-ellipse. The roots are given by
\begin{align*}
\nu_{1,2} &= \frac{a b c}{ab-ac+b c \pm2 \sqrt{a c (a-b) (c-b)}}.
\end{align*}
In the collared $\Hyp$-ellipse with $a=3, b=6, c=9$, the six real roots are $$\nu \in \left\{\;\frac{18}{11}, 6, \frac{18\pm 72 \sqrt{3}}{47}, \frac{198 \pm 36 \sqrt{13}}{23}\;\right\}.$$  In the case of the transverse $\Hyp$-ellipse with $a=3, b=-3, c=6$, the six real roots are $$\nu \in \left\{\;-\frac{6}{7}, 6, \frac{-30\pm 24 \sqrt{3}}{23}, \frac{-6 \pm 12 \sqrt{13}}{17}\;\right\}.$$ 

\begin{figure}[h]
\centering
\begin{tabular}{c c}
a) \includegraphics[width=0.45\textwidth]{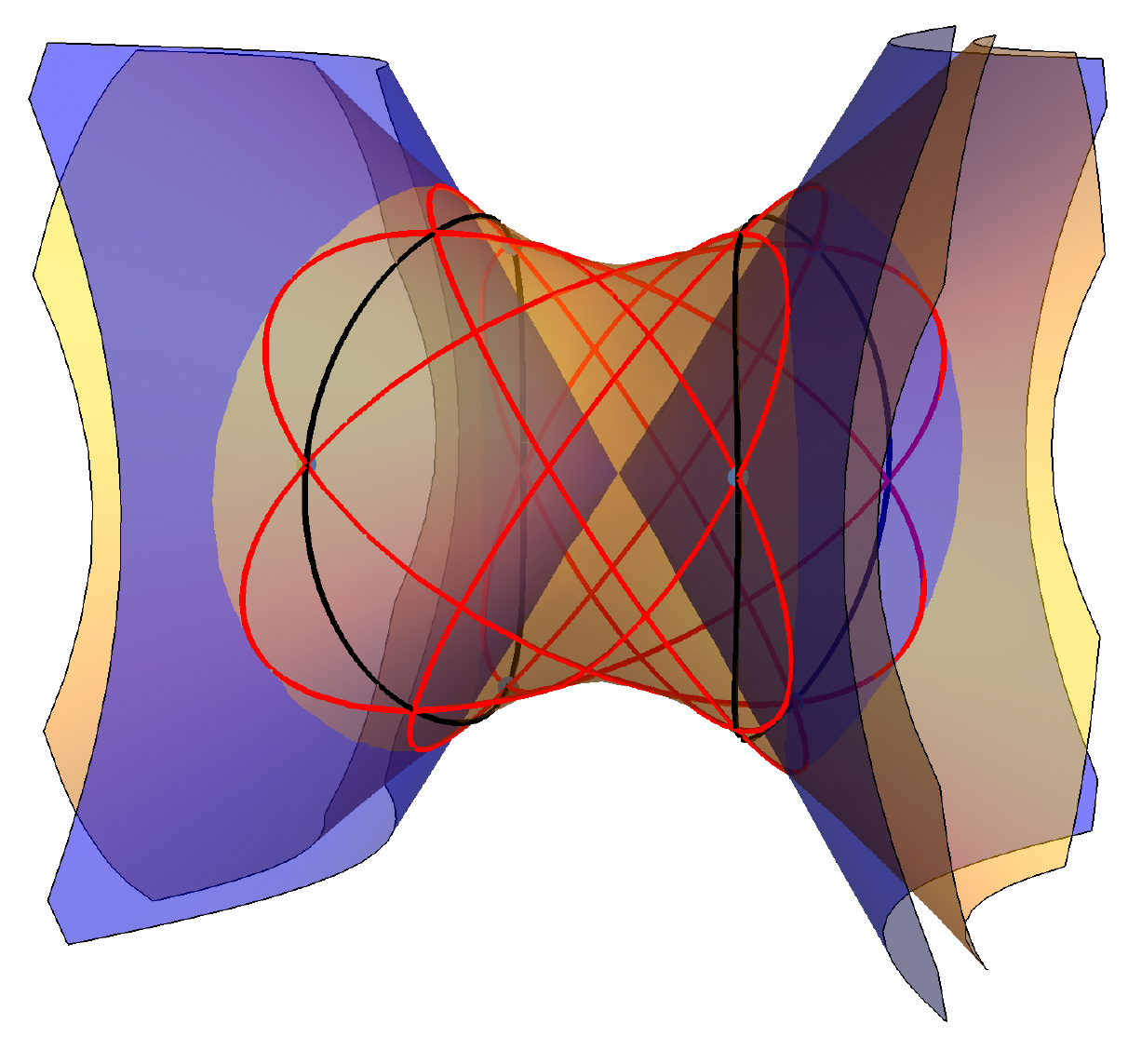} & b) \includegraphics[width=0.40\textwidth]{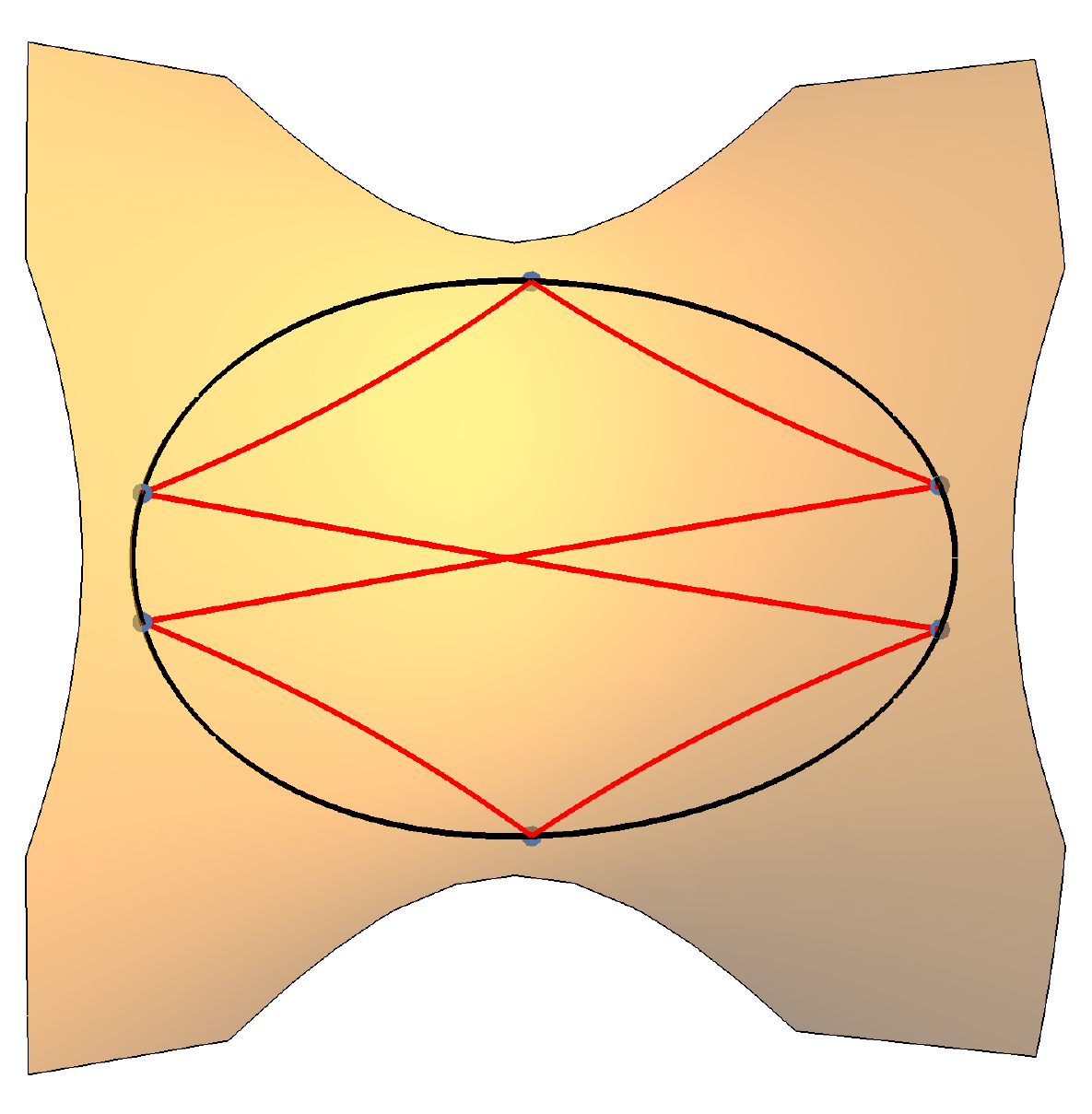} \\
\end{tabular}
\caption{(a) The collared $\Hyp$-ellipse and the space-like period 6 trajectory from figure \ref{EllipseExample} with the caustic cone corresponding to $\nu = (18-72 \sqrt{3})/47$. The geodesics from the billiard were extended to illustrate the tangency to the caustic.  (b) A time-like period 6 trajectory in the transverse $\Hyp$-ellipse showing symmetry across the coordinate planes $x_0=0$ and $x_1=0$.}
\label{PeriodicConfFam}
\end{figure}
Both of the cases include the degenerate conic corresponding to $\nu = 6$, though this conic is contained in different coordinate hyperplanes in each case. 

The condition for a light-like period 6 trajectory from (\ref{LightEvenCayley}) is that $E_3E_5-E_4^2=0$. This is equivalent to 
\begin{align*}
0&=\left(3 a^2 b^2-c^2 (b-a)^2-2 a b c (a+b)\right) \left(a^2 (c-b)^2-2 a b c (b+c)+b^2 c^2\right) \\
&\times \left(a^2 (c-b)^2+2 a b c (b+c)-3 b^2 c^2\right) \left(a^2 (b-c) (b+3 c)+2 a b c (c-b)+b^2 c^2\right)
\end{align*}
In the case of the collared $\Hyp$-ellipse, this has two solutions in terms of $a$, $b$, and $c$. One solution is $$(a,b,c) = \left(a, b, \frac{a b \left(2 \sqrt{b}+\sqrt{b-a}\right)}{(a+3 b)\sqrt{b-a}} \right) \text{ for } a < b < \frac{4a}{3}$$
and the other is 
$$ (a,b,c) = \left(a, b, \frac{a b \left(2 \sqrt{a b}+a+b\right)}{(b-a)^2} \right) \text{ for } a< b < 4a.$$
In the case of the transverse $\Hyp$-ellipse, there are two solutions in terms of $a$, $b$, and $c$. One solution is 
$$(a,b,c) = \left(a, \frac{ac\left(\sqrt{c-a}-2 \sqrt{c} \right)}{(a+3 c)\sqrt{c-a}}, c \right) $$
and the other solution is 
$$(a,b,c) = \left(a, -\frac{a c \left(2 \sqrt{a^2-a c+c^2}+a+c\right)}{(c-a)^2} ,c \right).$$
These conditions are equivalent to the cases when the denominators of $\nu_{1,2}$ above could possibly vanish. 

With the above conditions, a light-like period 6 trajectory will occur when the initial points $x$ and $y$ can be connected by a light-like geodesic that stays inside the collared or transverse $\Hyp$-ellipse. 

\begin{figure}[h]
\begin{tabular}{c c}
a) \includegraphics[width=0.45\textwidth]{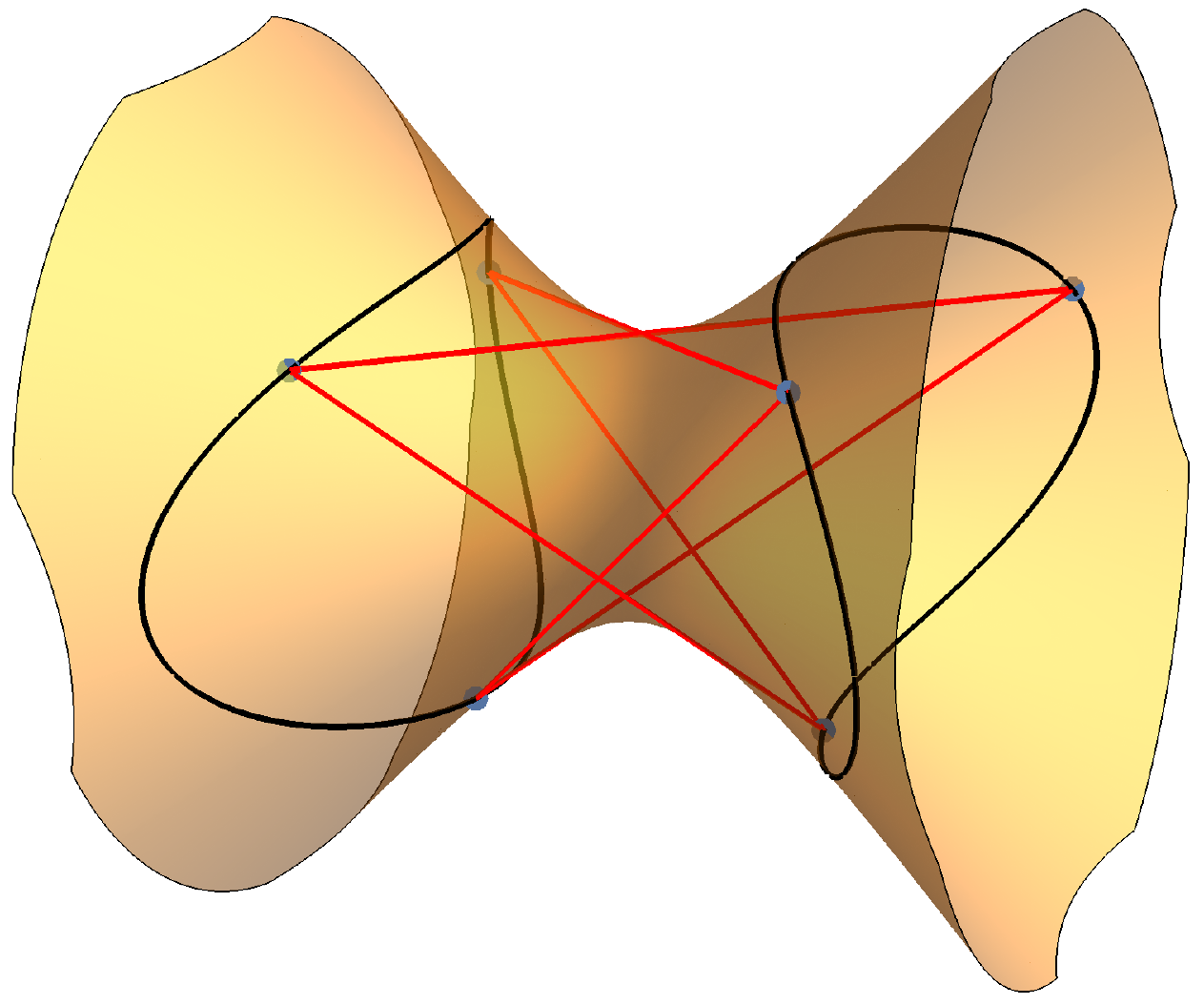} & b) \includegraphics[width=0.45\textwidth]{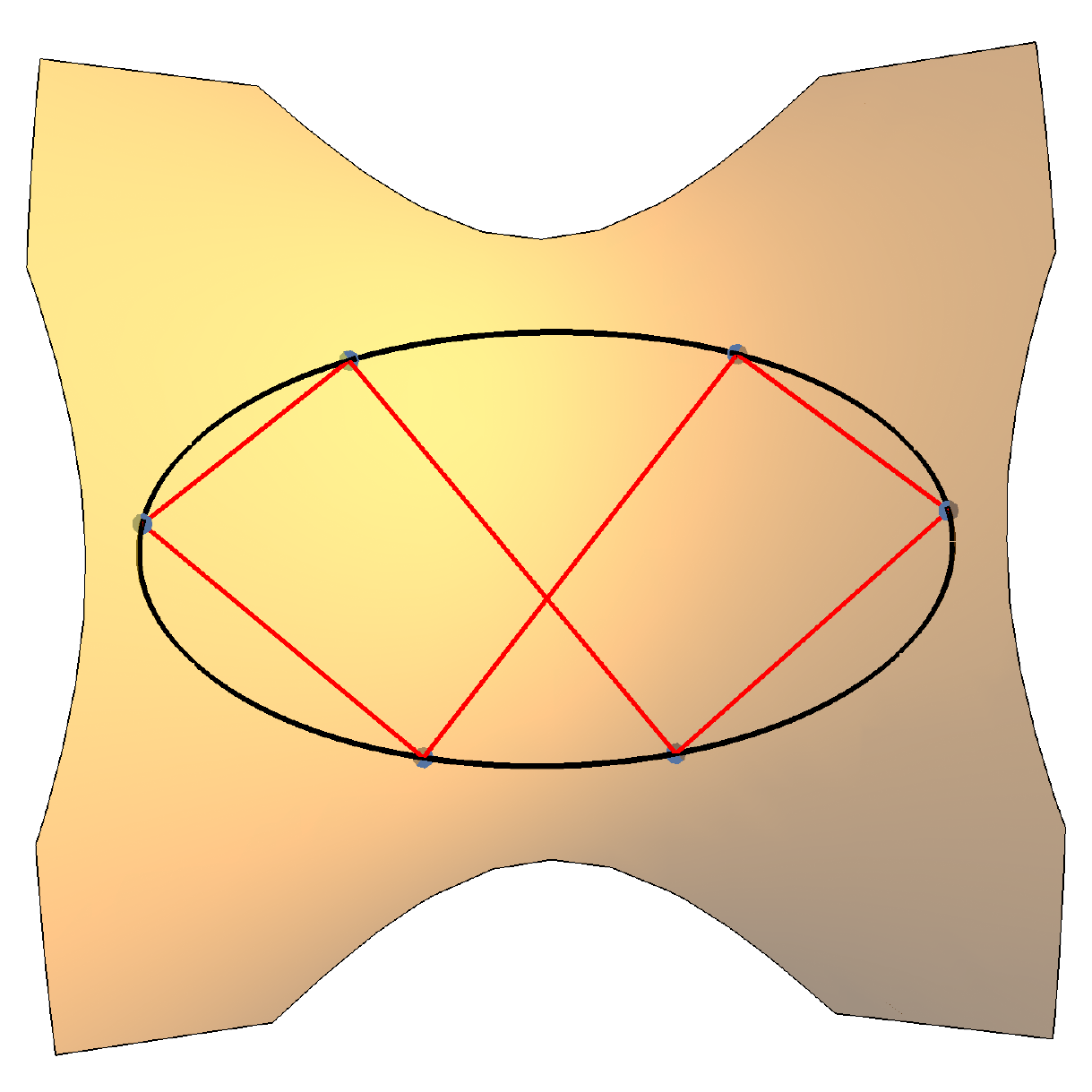}
\end{tabular}
\caption{Light-like period 6 orbits in the collared (a) and transverse (b) $\Hyp$-ellipse. }
\end{figure}
\end{example}

\section*{Acknowledgements}

This research is supported by the Discovery Project No.~DP190101838 \emph{Billiards within confocal quadrics and beyond} from the Australian Research Council.
Research of M.~R. is also supported by Mathematical Institute SANU, the Ministry of Education, Science and Technological Development of the Republic of Serbia, and the Science Fund of Serbia.

\bibliographystyle{amsalpha}
\nocite{*}
\bibliography{References1}

\smallskip

\hrule

\smallskip

\end{document}